\documentclass[a4paper,11pt]{article}

\usepackage{amssymb}
\usepackage{amsfonts}
\usepackage{amsmath}
\usepackage{color}

\usepackage{graphicx}
\usepackage[all]{xy}
\usepackage{array}
\usepackage{youngtab}

\usepackage{amsthm}

\newtheorem{theorem}{Theorem}[section]

\newtheorem{lemma}[theorem]{Lemma}
\newtheorem{proposition}[theorem]{Proposition}
\newtheorem{corollary}[theorem]{Corollary}

\theoremstyle{definition}
\newtheorem{definition}[theorem]{Definition}

\newtheorem{remark}[theorem]{Remark}
\newtheorem{example}[theorem]{Example}

\newtheorem*{acknowledgement}{Acknowledgement}

\newcommand{\To}{\longrightarrow}

\newcommand{\PP}{\mathcal{P}}

\newcommand{\V}{\mathcal{V}_k}

\newcommand{\X}{\mathfrak{X}}
\newcommand{\R}{\mathcal{R}}
\newcommand{\minus}{\smallsetminus}
\newcommand{\C}{\mathcal{C}}

\newcommand{\Schur}{\bold{S}}
\newcommand{\E}{\bold{E}}
\newcommand{\F}{\bold{F}}
\newcommand{\Z}{\mathbb{Z}}
\newcommand{\HH}{\mathcal{H}}
\newcommand{\g}{\mathfrak{g}}
\newcommand{\I}{\bold{I}}
\newcommand{\RT}{R\bold T^*}
\newcommand{\LT}{L^* \bold T}
\newcommand{\Frob}{ {(1)}}

\input xy
\xyoption{all}

\usepackage[normalem]{ulem}

\begin{document}

\title{Polynomial functors and  \\categorifications of Fock space II}

\author{Jiuzu Hong and Oded Yacobi}
\date{}

\maketitle

\begin{abstract}

We categorify various Fock space representations on the algebra of symmetric
functions via the category of polynomial functors.  In a
prequel,  we used polynomial functors to categorify  the Fock space representations
of type A
affine Lie algebras.  In the current work we continue the study of polynomial functors from the point of view of higher representation theory.  Firstly, we categorify the Fock space representation of the
Heisenberg
algebra on the category of polynomial functors.  Secondly, we construct commuting actions of the affine Lie algebra and the
level $p$ action of the Heisenberg algebra on the derived category of polynomial functors over a field of characteristic
$p>0$, thus weakly
categorifying the Fock space representation of $\widehat{\mathfrak{gl}}_p$.
Moreover, we study the relationship
between these categorifications and Schur-Weyl duality.
The duality is formulated as a functor from the category of polynomial functors
to the category of
linear species.   The category of linear species is known to carry actions
of the Kac-Moody algebra and the Heisenberg algebra.  We prove that Schur-Weyl
duality is a morphism of these categorification structures.

\end{abstract}


\section{Introduction}

It is oftentimes advantageous to study the symmetric group by considering properties of the family of all symmetric groups $\{\mathfrak{S}_n, n\geq0\}$.  This viewpoint already appears in the classical work of Zelevinsky, where the representation theory of the symmetric groups is developed from first principles via the natural Hopf algebra structure on the direct sum of their representation rings \cite{Zel}.

Sometimes completely new and surprising symmetries manifest from these considerations.  The direct sum of representation rings of symmetric groups is again an example of such a phenomenon.  Here a Kac-Moody algebra  $\g$ appears, which naturally acts on this algebra via so-called $i$-induction and $i$-restriction.

In recent years, focus has shifted away from the representation rings to the categories themselves.  For instance, instead of looking at the representation rings of the symmetric groups, we consider the \emph{category}
$$
\R=\bigoplus_{d\geq0}Rep(k\mathfrak{S}_d),
$$
where $Rep(k\mathfrak{S}_d)$ denotes the category of finite-dimensional representations of $\mathfrak{S}_d$ over a field $k$.  The aim is then to describe an action of $\g$ on this category.  Making this notion precise is subtle, and the representation theory of seemingly disparate objects, such as  affine Hecke algebras, must be invoked.

In fact, the category $\R$ has served as a motivating example for much of this theory.  In their landmark work, Chuang and Rouquier defined a notion of $\mathfrak{sl}_2$-categorification, and endowing $\R$ with such structure, proved Broue's abelian defect conjecture for the symmetric groups \cite{CR}.  This idea was later generalized  to Kac-Moody algebra of arbitrary type  in the works of Khovanov-Lauda \cite{KL1,KL2,KL3} and Rouquier \cite{R}.

We thus see that the category $\R$ is situated  at the center of important recent developments in categorification and representation theory.  Since the representations of symmetric groups and general linear groups are intimately related, it is natural to expect that an analogue of $\R$ for the general linear groups will also be closely connected to many categorification theories.  This analogue, namely the category $\PP$ of strict polynomial functors, is our central object of study.

In a prequel to this work we showed that $\PP$ is naturally endowed with a $\g$-action (in the sense of Chuang and Rouquier) categorifying the Fock space representation of $\g$ \cite{HTY}.  In the current work we continue the study of $\PP$ from the perspective of higher representation theory.  One of the main new ingredients that we emphasize  is that Schur-Weyl duality is a functor $\Schur:\PP\to\R$ which preserves various categorification structures.

The classical story that we study begins with the algebra $B$ of symmetric functions in infinitely many variables.  There are several interesting algebras  acting on $B$.  First of all, $B$ is the natural carrier for the Fock space representation of the Heisenberg algebra $H$.  Secondly, fixing a nonnegative integer $m$, $B$ carries also the Fock space representation of the affine Kac-Moody algebra $\g$, where
\begin{equation*}
\mathfrak{g}= \left\{
\begin{array}{lr}
\mathfrak{sl}_{\infty} \text{   if } m=0  &  \\
\widehat{\mathfrak{sl}}_{m} \text{  if }  m>0
\end{array}
\right.
\end{equation*}
When $m=0$ this is an irreducible representation, but for $m>0$ this is no longer the case.  In the case $m>0$, there is a copy of the Heisenberg algebra acting on $B$ which commutes with the action of $\g'=[\g,\g]$, which we term the ``twisted'' Heisenberg algebra and denote $H^{(1)}$.  Then in this case $B$ is an irreducible $U(\g')\otimes H^{(1)}$-module.  In fact, the $U(\g')\otimes H^{(1)}$-action can be viewed as an action of $U(\widehat{\mathfrak{gl}}_m)$ on $B$.

It is well known that $\PP$ categorifies the algebra $B$ of symmetric functions, and, as mentioned above, we've shown in previous work that $\PP$ categorifies the action of $\g$ on $B$.   Our aim in the current work is to show $\PP$ also categorifies the other representations on $B$ mentioned in the previous paragraph.  We show that $\PP$ categorifies the $H$ action on $B$ in the sense of Khovanov, and that the derived cateogry of $\PP $ weakly categorifies the $H^{(1)}$ action on $B$. We also prove that the actions of $H^{(1)}$ and $\g'$ on $D^b(\PP)$ commute, thus categorifying the $U(\g')\otimes H^{(1)}$ action on $B$, i.e. the Fock space representation of $\widehat{\mathfrak{gl}}_p$. Moreover, we  show that the Schur-Weyl duality functor $\Schur:\PP\to\R$ is naturally compatible with all the categorification structures that we can endow on both categories.  That is, $\Schur$ can be enriched to a morphism of both $\g$-categorifications and $H$-categorifications.  Let us now describe contents of the paper in more detail.

 In Section \ref{SectionClassical} we review the classical picture, and, in particular, the representations that we categorify in this work.  We recall a less familiar presentation of the Heisenberg algebra from Khovanov's work.

Our main object of study, the category $\PP$, is introduced in Section \ref{SectionP}.  From this section on we fix an infinite field $k$ of characteristic $p$.  We let $\V$ denote the category of finite dimensional vector spaces over $k$.  Then $\PP$ is a functor category whose objects are polynomial  endofunctors on $\V$ which satisfy finiteness conditions.

When $p=0$ the category $\PP$ already appears in MacDonald's book \cite{Mac}.  In positive characteristic, and in greater generality, $\PP$ is studied in the works of Friedlander and Suslin, where they use this category to prove the finite generation of the cohomology of finite group schemes \cite{FS97}.  Polynomial functors have subsequently found many applications in algebraic topology (see p.2 of $\cite{HTY}$ for a brief discussion and references on these applications).

In Section \ref{SectionR} we introduce the category $\R$.  Instead of using as a model for $\R$ the representations of  symmetric groups, it is more natural for us to define $\R$ using Joyal's theory of linear species.  Thus, objects in $\R$ are functors from the groupoid of finite sets to $\V$ which satisfy some finiteness conditions.  (Some authors refer to these as ``vector species''.  A modern account of the theory of linear species appears in \cite{AM}.)  The category $\R$  is naturally equivalent to the direct sum of representations of symmetric groups introduced above, but is more amenable to the higher structures we are interested in.

Schur-Weyl duality enters the picture in Section \ref{SectionS}, where we define a tensor functor $\Schur:\PP\to\R$.  We will see throughout the course of this work that $\Schur$ is compatible with all the higher structures inherent to both categories.

Our study of categorifications begins in Section \ref{Sectiong}, where we consider the theory of $\g$-categorifications.  First we review Rouquier's definition of $\g$-categorification and morphisms of such structures \cite{R}.  The definition of $\g$-categorification consists of functors on some category $\C$ that descend to an action of $\g$ on the Grothendieck group, and   isomorphisms between these functors and satisfying certain conditions.  Subsequently, a morphism of $\g$-categorifications has to preserve all these higher structures in a suitable sense.

After this review of  Rouquier's theory, we briefly recall the $\g$-categorification on $\PP$, which categorifies the Fock space representation of $\g$.  This result appears in the prequel to this work \cite{HTY}.  Next we recall the $\g$-categorification on $\R$ which categorifies the basic representation of $\g$.  Here everything is stated in the language of linear species, although it is equivalent to the well-known formulation in terms of representations of symmetric groups.    Finally, we conclude Section 6 with the first new theorem in this work (Theorem \ref{Thm_Smorphgcat}).  This theorem states that $\Schur:\PP\to\R$ is a morphism of $\g$-categorifications which categorifies a projection from Fock space onto the basic representation of $\g$. In other words, $\Schur$ not only induces an operator on the Grothendieck groups which is $\g$-equivariant, it also preserves all the higher structures on the categories themselves.

In Section \ref{SectionHCat} we move on to consider $H$-categorifications.  First we review Khovanov's category $\HH$ \cite{Kh}. We opt to define the category algebraically, although in Khovanov's original work it is defiend graphically.   In characteristic zero Khovanov shows that the Grothendeick group of $\HH$ contains $H_\Z$, an integral form of $H$.

Using this category, one can define a notion of a ``strong'' $H$-action on linear categories defined over characteristic zero.  We introduce also a notion of a morphism of $H$-categorifications, analogous to Rouquier's definition of morphism of $\g$-categorifications.   Morphisms of $H$-categorifications categorify morphisms of $H$-modules (Proposition \ref{PropHmorph}).  In positive characteristic, Khovanov's theory does not provide a notion of ``strong'' $H$-categorification, and so in this case the defining relations of $H_\Z$ define a notion of weak $H$-categorification.  (In forthcoming work, we extend the Khovanov and Licata-Savage theory of  $H$-categorification to categories defined over positive characteristic \cite{HYII}.)

Having reviewed the general theory of Heisenberg categorification, we prove our second main result (Propositions \ref{PropWellDef} and \ref{PropIdem}), which  is that $\HH$ naturally acts on $\PP$.  More precisely, we construct a natural functor $\HH \to End_L(\PP)$, where $End_L(\PP)$ is the category of  endo-functors on $\PP$ admitting left adjoint functors.  When $p=0$ we show that this gives an  $H$-categorification of the Fock space representation (Theorem \ref{Thm_HcatonP}).  When $p>0$ we get a weak $H$-categorification of the Fock space representation (Theorem \ref{Thm_Hweakcat}).  In the forthcoming paper $\cite{HYII}$ this categorification is enhanced to a strong $H$-categorification.

Next, in Section \ref{SectionHCatonR}, we briefly recall Khovanov's $H$-categorification of the Fock space representation.  Again, we state this result in the language of linear species which is more suitable for our purposes.  We then prove that when $p=0$ the Schur-Weyl duality functor $\Schur:\PP\to\R$ is an equivalence of $H$-categorifications (Theorem \ref{Thm_SequivHcat}).

The virtue of having a weak $H$-categorification when $p>0$ is that it leads us  to categorify the twisted $H^{(1)}$-action on $B$, which is taken up in Section \ref{SectionTwist}.  Recall that when $p>0$, there is a copy of the Heisenberg algebra acting on $B$ which commutes with the $\g'$-action.  Using the Frobenius twist functor on $\PP$ we categorify the endomorphisms on $B$ commuting with the $\g'$-action.  More precisely, we prove Theorems \ref{Commutativity_Theorem} and \ref{Adjunction_Commuting_Action}, which state that the endofunctors on $\PP$ given by tensoring by Frobenius twists of polynomial functors (and their adjoints) are   morphisms of $\g'$-categorifications on $\PP$. These endomorphisms of the $\g'$-categorification on $\PP$  categorify the $H^{(1)}$-action on $B$, but, in general, the adjoint functors are not exact.  This forces us to move to the bounded derived category $D^b(\PP)$.  In Theorem \ref{TheoremCommutingAction} we describe a $U(\g')\otimes H^{(1)}$-action on $D^b(\PP)$, which categorifies the irreducible $U(\g')\otimes H^{(1)}$-action on $B$.  We note in Remark \ref{Remarkglhat} that this is an instance of $\widehat{\mathfrak{gl}}_p$-categorification.

In the Appendix we collect some technical  results which are necessary for the main part of the paper. In particular  we  prove Proposition \ref{Adjunction_Morphism_Categorification} which states that an adjoint of a morphism of $\g$-categorifications (suitably enriched)  is also a morphism of $\g$-categorifications.

\begin{acknowledgement}
We are grateful to A. Touz\'e  and P. Tingley for helpful discussions at the early stage of this work, and to J. Bernstein for his encouragement and  insights throughout the course of this work.
J.H. would like to also thank M. Fang for some useful discussions on Schur-Weyl duality.
\end{acknowledgement}

\section{The classical picture}
\label{SectionClassical}

We begin by reviewing various  structures related to the algebra of symmetric functions.

\subsection{The algebra $B$}
\label{ThealgebraB}
Let $\mathfrak{S}_n$ denote the symmetric group on $n$ letters.  $\mathfrak{S}_n$ acts on the polynomial algebra $\Z[x_1,...,x_n]$ by permuting variables, and we denote by $B_n=\Z[x_1,...,x_n]^{\mathfrak{S}_n}$  the polynomials invariant under this action.  There is a natural projection
$
B_n \twoheadrightarrow B_{n-1}
$
given by setting the last variable to zero.  Consequently, the rings $B_{n}$ form an inverse system; let $B_\Z$ denote the subspace of finite degree elements in the inverse limit
$\varprojlim B_n$.  This is the \emph{algebra of symmetric functions} in infinitely many variables $\{x_1,x_2,... \}$.  Let $B=B_\Z\otimes_\Z\mathbb{C}$ denote the  \emph{(bosonic) Fock space}.

The algebra $B_\Z$ has many  well-known bases.  Perhaps the nicest is the basis of \emph{Schur functions} (see e.g. \cite{Mac}).  Let $\wp$ denote the set of all partitions, and for $\lambda\in\wp$ let $s_\lambda\in B_\Z$ denote the corresponding Schur function.
These form a $\Z$-basis of the algebra of symmetric functions:
$$
B_\Z=\bigoplus_{\lambda\in\wp}\Z s_{\lambda}.
$$

\subsection{The Heisenberg action on $B$}
Let $(\cdot,\cdot)$ be the inner product on $B$ defined by declaring the Schur functions to be orthonormal: $(s_\lambda,s_\mu)=\delta_{\lambda,\mu}$.  Any $b\in B$ defines an operator on $B$ via multiplication which we continue to denote by $b:B\to B$.  Let $b^*:B\to B$ denote the adjoint operator with respect to the inner product.

\begin{definition}
The \emph{Heisenberg algebra} $H$ is the subalgebra of $End(B)$ generated by $b,b^*$ for all $b\in\ B$.  The natural representation of $H$ on $B$ is the \emph{Fock space representation} of $H$.
\end{definition}

The Fock space representation of $H$ is irreducible, and, up to isomorphism, it is the unique irreducible (level 1) representation of $H$ (cf. \cite{Ka}, \cite{Lec}).

The definition of $H$ above is redundant, in the sense that of course we don't need to take \emph{all} $b \in B$ to define $H$.  It's enough to take some algebraically independent generating set, and different choices of such sets lead to different presentations of $H$.  This point will become important later when we discuss categorifications of $H$-modules.  Here are some examples.

\begin{example}
\label{example_H}
\begin{enumerate}
\item Let $p_r$ be the $r^{th}$ power symmetric function.  Then the operators $p_r$ and $p_r^*$ on $B$ satisfy the following relations
\begin{eqnarray*}
p_rp_s&=&p_sp_r \\ p_r^*p_s^* &=& p_s^*p_r^* \\ p_r^*p_s &=& p_sp_r^* +s\delta_{r,s}1
\end{eqnarray*}
This leads to the usual presentation of the Heisenberg algebra $H$.  In other words, we can take as a definition for $H$ to be the unital algebra generated by symbols $p_r$ and $p_r^*$ subject to the above relations.

\item Let $e_r$ be the $r^{th}$ elementary symmetric function, and $h_r$ the $r^{th}$ complete  symmetric function.   Then operators  $e_r$ and $h_r^*$ on $B$ satisfy the following relations
\begin{eqnarray*}
e_re_s&=&e_se_r \\ h_r^*h_s^* &=& h_s^*h_r^* \\ h_r^*e_s &=& e_sh_r^* +e_{s-1}h_{r-1}^*.
\end{eqnarray*}
This gives another presentation of $H$.

\end{enumerate}
\end{example}

\begin{definition}
Let $H_\Z$ denote the unital ring over $\Z$ with generators $e_r,h_r^* $ (as symbols),$r\geq1$ subject to the relations as in part two of the above example.  The ring $H_\Z$ is an \emph{integral form of $H$}, i.e. $H_\Z\otimes_\Z\mathbb{C}\cong H$ (see e.g. \cite{Kh}).  We refer to the action of $H_\Z$ on $B_\Z$ as the \emph{Fock space representation} of $H_\Z$.
\end{definition}

\subsection{More symmetries on $B$}
\label{MoresymmetriesonB}

Fix an integer $m\geq0$.  Let $\mathfrak{g}$ denote the following complex Kac-Moody algebra:
\begin{equation*}
\mathfrak{g}= \left\{
\begin{array}{lr}
\mathfrak{sl}_{\infty} \text{   if } m=0  &  \\
\widehat{\mathfrak{sl}}_{m} \text{  if }  m>0
\end{array}
\right.
\end{equation*}
For the  precise relations defining $\g$ see \cite{Ka}. The Lie
algebra $\g$ has standard Chevalley generators $\{e_i,f_i\}_{i \in
\Z /m\Z}$.

Set $h_i=[e_i,f_i]$. The Cartan subalgebra of $\g$ is spanned by the $h_i$ along with an element $d$.  Let $\{\Lambda_i:i \in
\Z/m\Z\}\cup\{\delta\}$ be the dual weights.  We let $Q$ denote the root lattice and $P$ the
weight lattice of $\g$.  Let $V(\Lambda)$ be the irreducible integrable $\g$-module
of highest weight $\Lambda \in P$ (so $\Lambda$ is necessarily dominant integral).  In particular, the \emph{basic
representation} is $V(\Lambda_0)$.  Let $\g'=[\g,\g]$ be the derived algebra of $\g$.

Let us review some combinatorial notions related to Young diagrams. Firstly, we identify partitions with their Young diagram (using English
notation).  For example the partition $(4,4,2,1)$ corresponds to the
diagram
 $$
 \yng(4,4,2,1)
 $$
The \emph{content} of a box in position $(k,l)$ is the integer $l-k
\in \Z/m\Z$.  Given $\mu,\lambda \in \wp$, we write
$
\xymatrix{ \mu \ar[r] & \lambda}
$
if $\lambda$ can be obtained from $\mu$ by adding some box.  If the
arrow is labelled $i$ then $\lambda$ is obtained from $\mu$ by
adding a box of content $i$ (an $i$-box, for short).  For instance,
if $m=3$, $\mu=(2)$ and $\lambda=(2,1)$ then
$
\xymatrix{
 \mu \ar[r]^{2} & \lambda}.
$
An $i$-box of $\lambda$ is \emph{addable} (resp. \emph{removable})
if it can be added to (resp. removed from) $\lambda$ to obtain
another partition.

Of central importance to us is the \emph{Fock space} representation
of $\g$ on $B$ (or $B_\Z$).  The action of $\g$ on $B$ is given by the following formulas:
$
e_i.s_{\lambda}=\sum s_{\mu}
$,
the sum over all $\mu$ such that $\xymatrix{ \mu \ar[r]^{i} &
\lambda}$, and
 $
f_i.s_{\lambda}=\sum s_{\mu}
$,
the sum over all $\mu$ such that $\xymatrix{ \lambda \ar[r]^{i} &
\mu}$.
Moreover,  $d$ acts on $s_\lambda$ by $m_0(\lambda)$, where $m_0(\lambda)$ is the number of boxes of content zero in $\lambda$.  These equations define an integral  representation of
$\mathfrak{g}$.(see e.g. \cite{Lec},\cite{LLT})..


The basic representation $V(\Lambda_0)$ is isomorphic to $U(\g).1\subset B$, i.e. the $\g$-submodule of $B$ generated by the unit.  We identify $V(\Lambda_0)$ as a submodule of $B$\ in this manner.

The Fock space representation of $\g$ is  semisimple, and the basic representation occurs with multiplicity one in $B$.  Hence there is a natural projection onto $V(\Lambda_0)$, $\pi: B \twoheadrightarrow V(\Lambda_0)$, which we term the \emph{standard projection}.

 When $m=0$, $V(\Lambda_0)=B$, but when $m>0$ this is no longer the case.  Indeed suppose $m>0$.  For $b\in B$, let $b^{(1)}\in B$ be defined by $b^{(1)}(x_1,x_2,...)=b(x_1^m,x_2^m,...)$.  We refer to $b^{(1)}$ as the \emph{twist} of $b$.  Let $H^{(1)}\subset End(B)$ be the subalgebra generated by $b^{(1)},(b^{(1)})^*$ for all $b\in B$. We refer to this as the ``twisted'' Heisenberg algebra.  The following lemma is well-known, and can be found e.g. in Section 2.2.8 of \cite{Lec}.

\begin{lemma}
\label{Lemma_CommutingH}
\begin{enumerate}
\item As algebras, $H^{(1)} \cong H$.
\item The $H^{(1)}$ action on $B$ commutes with the $\g'$ action on $B$.
\item As a $U(\g')\otimes H^{(1)}$-module $B$ is irreducible, moreover $End_{\g'}(B)\simeq H^{(1)}$.
\end{enumerate}
\end{lemma}

\begin{remark}
The action of $U(\g')\otimes H^{(1)}$ on $B$ is equivalent to the action of  $U(\widehat{\mathfrak{gl}}_n')$ on $B$, where $\widehat{\mathfrak{gl}}_n'$ is the derived algebra of $\widehat{\mathfrak{gl}}_n$ (i.e. $\widehat{\mathfrak{gl}}_n'=\mathfrak{gl}_n\otimes \mathbb{C}[t,t^{-1}]\oplus\mathbb{C}c$).  For our purposes it is more natural to use the tensor product realization since we categorify these commuting actions sparately.
\end{remark}

To summarize, we've introduced the algebra $B$, and considered the following structures

on it: the (irreducible) $H$-action on $B$, the $\g$-action on $B$, and the (irreducible) $U(\g')\otimes H^{(1)}$-action on $B$ for  $m>0$.  One of our main goals is to show that the category $\PP$ of polynomial functors (which we introduce below) naturally categorifies all this structure when $m=char(k)$ is  the characteristic of some field $k$.

\section{Strict polynomial functors}
\label{SectionP}
In this section we define the category $\PP$ and recall some of its basic features.  Our basic reference is \cite{FS97}; see also \cite{HTY} for a more concise recollection of the category $\PP$ and its  properties that we utilize.

For the remainder of the paper, fix an infinite field $k$ of characteristic $p\geq0$.  The characteristic $p$ plays the role of the fixed nonnegative integer $m$ of the previous section.  So for example, from now on anytime we refer to the Kac-Moody algebra $\g$, it is either $\mathfrak{sl}_\infty $ or $\widehat{\mathfrak{sl}}_p$ depending on whether $p=0$ or $p>0$.

For a $k$-linear abelian or triangulated category $\C$, let $K_0(\C)$ denote the Grothendieck group of $\C$, and let $K(\C)$ denote the complexification of $K_0(\C)$.  If $A\in\C$ we let $[A]$ denote its image in $K_0(\C)$. Simiarly, for an exact functor $F:\C\to\C'$ we let $[F]:K_0(\C)\to K_0(\C')$ denote the induced operator on the Grothendieck groups.  Slightly abusing notation, the complexification of $[F]$ is also denoted by $[F]$.

 Let $\V$ denote the category of finite dimensional vector spaces over $k$.  For $V,W \in \V$, \emph{polynomial maps} from $V$ to $W$ are by definition elements of $S(V^{*})\otimes W$, where $S(V^*)$ denotes the symmetric algebra of the linear dual of $V$.  Elements of $S^{d}(V^{*})\otimes W$ are said to be \emph{homogeneous} of degree $d$.

\subsection{The category $\PP$}

The objects of the category $\PP$ are strict polynomial functors of finite degree, i.e. functors $M:\V \rightarrow \V$ that satisfy the following two properties: for any $V,W \in \V$, the map of vector spaces
$$
Hom_{k}(V,W) \rightarrow Hom_{k}(M(V),M(W))
$$
is polynomial, and the degree of the map
$
End_{k}(V) \rightarrow End_{k}(M(V))
$
is uniformly bounded for any $V \in \V$.
The morphisms in $\PP$ are natural transformations of functors.  For $M \in \PP$ we denote by $1_M\in Hom_{\PP}(M,M)$ the identity natural transformation.  $\PP$ is a $k$-linear abelian category.

Let $I \in \PP$ be the identity functor from $\V$ to $\V$ and let $k\in \PP$ denote the constant functor with value $k$.
Tensor products in $\V$ define a tensor structure $\otimes$ on $\PP$, with unit $k$.

The $d$-fold tensor product, $\otimes^d$, is an example of a polynomial functor.  Similarly, the exterior and symmetric powers, $\Lambda^d$ and $S^d$, are also polynomial functors.

The $d^{th}$ \emph{divided power}, $\Gamma^d$, is a polynomial functor that plays an important role.  It is defined on $V\in\V$ by $\Gamma^d(V)=(\otimes^dV)^{\mathfrak{S}_d}$, the $\mathfrak{S}_d$-invariants in $\otimes^dV$.  When $p=0$ then $\Gamma^d \cong S^d$, but in positive characteristic this is no longer the case.

For $W\in\V$ we also define $\Gamma^{W,d}\in\PP$ by $\Gamma^{W,d}(V)=\Gamma^d(Hom(W,V))$.  Similarly, $S^{W,d}$ is given by $S^{W,d}(V)=S^d(Hom(W,V))$.

Finally, when $p>0$ we have the \emph{Frobenius twist} functor $(\cdot)^{(1)}\in\PP$ which maps $V\in\V$ to $V^{(1)}$, the Frobenius twist of $V$.  (Recall that $V^{(1)}=k\otimes_\phi V$, where $\phi:k\to k$ maps $z \mapsto z^p$.)

Let $M \in \PP$ and $V \in \V$.  By functoriality, evaluation on $V$ yields a functor from
$\PP$ to  $Pol(GL(V))$, the category of polynomial representations of $GL(V)$.

The \emph{degree} of a functor $M\in\PP$ is the upper bound of the degrees of the polynomials $End_k(V)\to End_k(M(V))$ for $V\in\V$. A functor $M\in\PP$ is \emph{homogeneous of degree $d$} if all the polynomials $End_k(V)\to End_k(M(V))$ are homogeneous polynomials of degree $d$.

The category $\PP$ splits as the direct sum of its subcategories $\PP_d$ of homogeneous functors of degree $d$:
\begin{equation}
\label{Pdecomp}
\PP=\bigoplus_{d\ge 0} \PP_d\;.
\end{equation}
Any $M\in\PP$ can thus be expressed as $M=\oplus M_d$, where $M_d$ is homogeneous of degree $d$ \cite[Proposition 2.6]{FS97}.

If $M\in\PP$, we define its \emph{Kuhn dual} $M^\sharp\in\PP$ by  $M^\sharp(V)=M(V^*)^*$, where `$*$' refers to $k$-linear duality in the category of vector spaces. Since $(M^{\sharp})^\sharp\simeq M$, Kuhn duality yields an equivalence of categories
$^\sharp:\PP\xrightarrow[]{\simeq} \PP^{\mathrm{op}}$, op. cit.
  A routine check shows that $^\sharp$ respect degrees, i.e. $M^\sharp$ is homogeneous of degree $d$ if and only if $M$ is homogeneous.

The following lemma shows  the category $\PP $ is a model for the stable category of  polynomial $GL_n$-modules.  Let $Pol_d(GL(V))$ denote the category of polynomial representations of $GL(V)$ of degree $d$.
\begin{lemma}[\cite{FS97}, Lemma 3.4]\label{thm-FS}
Let $V\in\V$ be a vector space of dimension $n\ge d$. Evaluation on $V$ induces an equivalence of categories
$\PP_d\to Pol_d(GL(V)).$
\end{lemma}

It is classically known that the representation ring of  $Pol_d(GL(V))$ is $B_d$, when $dim(V)\geq d$ (cf. Section \ref{ThealgebraB}).  This suggests that one can view the tensor category $\PP$ as a categorification of $B$, and this is the perspective we take in the current work.  Let us make this explicit.

For $\lambda \in \wp$ let $S_\lambda \in \PP$ denote the corresponding \emph{Schur functor}.  For precise definitions and more details on the following proposition see Section 4.3 of \cite{HTY}  and references therein.

\begin{proposition}
The map $\varrho:K_0(\PP)\to B_\Z$ determined by $[S_\lambda]\mapsto s_\lambda$ defines an algebra isomorphism.
\end{proposition}

We note  that the Euler characteristic of the derived Hom functor on $\PP$ categorifies the standard inner product on $B$.  More precisely, let $\left< \cdot , \cdot \right>$ be the bilinear form on $K_0(\PP)$ given by
$$
\left<[M],[N]\right>=\sum(-1)^idimExt^i(M,N).
$$
Then, under the identification $\varrho$, $\left< \cdot , \cdot \right>$ corresponds to $(\cdot, \cdot)$ (for details see \cite[Propsition 6.4]{HTY}).

Let $M\in\PP$.  Tensor product defines a functor $\bold{T}_M:\PP\to\PP$ given by $\bold{T}_M(N)=M\otimes N$.  The functor $\bold{T}_M$ admits the right adjoint  $\bold{T}^*_M$, which is given by a general formula as follows: for  $N\in\PP$ and $V\in\V$,
$$\bold{T}^*_M(N)(V)=Hom_{\PP}(M,N(V\oplus\cdot)).
$$
Here $N(V\oplus\cdot)\in\PP$ denotes the functor $W\mapsto N(V \oplus W)$.  Therefore we obtain the following lemma (originally communicated to us by Antoine Touz\'e):
\begin{lemma}
Let $M\in\PP$.  Then $\bold{T}_M^*$ is exact if and only if $M$ is projective.
\end{lemma}
Now suppose $\varrho([M])=b$.  Then, under the identification $\varrho$, $[\bold{T}_M]=b$, where here we regard $b$ as an operator on $B$. Recall that $b^*$ denotes the operator on $B$ adjoint to $b$.  When $M$ is projective, by the above lemma, $[\bold{T}_M^*]$ defines an operator on $K_0(\PP)$.  Then $[\bold{T}_M^*]$ is adjoint to $[\bold{T}_M]$ with respect to $\left<\cdot,\cdot\right>$, and, since $\left< \cdot , \cdot \right>$ corresponds to $(\cdot, \cdot)$, we obtain:
\begin{lemma}
\label{LemmaAdjoint}
Let $M\in\PP$ be projective and suppose $\varrho([M])=b$.  Then, under the identification $\varrho$, $[\bold{T}^*_M]=b^*$.
\end{lemma}
Finally, we note that the Frobenius twist can also be regarded as a \emph{functor} $(\cdot)^{(1)}:\PP\to\PP$, which assigns $M$ to $M^{(1)}$, where $M^{(1)}(V)=M(V)^{(1)}$.  If $[M]\mapsto b$ under $\varrho$, then $[M^{(1)}]\mapsto b^{(1)}$ (thus explaining why we termed $b^{(1)}$ the twist of $b$).  This is an example of a more general phenomenon, namely that composition of polynomial functors categorifies the plethysm of symmetric functions.



\subsection{Polynomial bifunctors}\label{SectionBi}
We shall also need the category $\PP^{[2]}$ of \emph{ strict polynomial bi-functors}. The objects of $\PP^{[2]}$ are functors $B:\V\times\V\to \V$ such that for every $V\in\V$, the functors $B(\cdot,V)$ and $B(V,\cdot)$ are in $\PP$ and their degrees are bounded with respect to $V$. Morphisms in $\PP^{[2]}$ are natural
transformations of functors.


A bifunctor $B$ is \emph{homogeneous of bidegree} $(d,e)$ if for all $V\in\V$, $B(V,\cdot)$ (resp. $B(\cdot,V)$) is a homogeneous strict polynomial functor of degree $d$, (resp. of degree $e$). The decomposition of strict polynomial functors into a finite direct sums of homogeneous functors generalizes to bifunctors. Indeed, $\PP^{[2]}$ splits as the direct sum of its full subcategories $\PP^{[2]}_{i,j}$ of homogeneous bifunctors of bidegree $(i,j)$. If $B\in\PP^{[2]}$, we denote by $B_{*,j}$ the direct sum:
\begin{equation}
\label{Bdecomp}
B_{*,j}=\bigoplus_{i\ge 0}B_{i,j}
\end{equation}
Note that we have also a duality for bifunctors
$$^\sharp:\PP^{[2]}\xrightarrow[]{\simeq} \PP^{[2]\,\mathrm{op}}$$
which sends $B$ to $B^\sharp$, with $B^\sharp(V,W)=B(V^*,W^*)^*$, and which respects the bidegrees.


There are natural functors:
\begin{align*}
&\boxtimes:\PP\times\PP\to \PP^{[2]} &&\otimes:\PP\times\PP\to \PP\\
&\nabla:\PP^{[2]}\to \PP && \Delta:\PP\to \PP^{[2]}
\end{align*}
respectively given by
\begin{align*}
&(M\boxtimes N)(V,W)=M(V)\otimes N(W) &&(M\otimes N)(V)= M(V)\otimes N(V)\\
&\nabla B (V)=B(V,V) && \Delta M(V,W)=M(V\oplus W)
\end{align*}

We conclude this section by introducing a construction of new functors in $\PP$ from old ones that will be used later.  Let $M\in\PP$ and consider the functor $\Delta M(\cdot,k)\in\PP$. By (\ref{Bdecomp}) we have a decomposition
$$
\Delta M(\cdot,k)=\bigoplus_{i\geq0} \Delta M_{*,i}(\cdot,k).
$$
Note that $\Delta M_{*,i}(V,k)$ is the subspace of weight $i$ of $M(V\oplus k)$ acted on by $GL(k)$ via the composition
$$GL(k)=1_V\times GL(k)\hookrightarrow GL(V\oplus k)\xrightarrow[]{} GL(M(V\oplus k))\;.$$
In other words, the action of $GL(k)$ induces a decomposition of $M(V\oplus k)$ indexed by the polynomial characters of $GL(k)$
$$
M(V\oplus k)=\bigoplus_{i\geq0}M(V\oplus k)_i,
$$
and $\Delta M_{*,i}(V,k)=M(V\oplus k)_i$.

Since evaluation on $V\oplus k$ and extracting weight spaces are both exact functors, the assignment $M \mapsto \Delta M_{*,i}(\cdot,k)$ defines an exact endo-functor on $\PP$.

\section{Linear species}
\label{SectionR}

\subsection{Finite sets}
It will be convenient for us to use as a model for the category of representations of all symmetric groups Joyal's theory of linear species \cite{AM}.
Let $\X$ be the groupoid of finite sets, and let $\X_k$ be the sub-groupoid of finite sets of cardinality $k$.  So, by definition, the morphisms in these categories are bijections.  Clearly, $\X=\bigcup_k \X_k$, and $\X_k$ is a connected groupoid with stabilizer isomorphic to the symmetric group $\mathfrak{S}_k$.

We fix once and for all a one-element set $\{*\}$.  We abuse notation and refer to the set $\{*\}$ simply by $*$, and we employ this convention for other one-element sets as well.

If $J,J',K,K'\in \X$ and $f:J\to K,f':J'\to K'$ are bijections, then
$$
f\sqcup f':J\sqcup J' \to K \sqcup K'
$$
is the canonical bijection on the disjoint unions.

Let us now introduce some  morphisms in $\X$. For $J\in\X$, $1_J:J\to J$ denotes the identity map.
To a bijection $f:J\to K$ and $j\in J$, we associate the map $f^j:J\minus j \to K\minus {f(j)}$.

For $i,j\in J$,  $s_{i,j}:J\to J$ is the ``transposition'':
\begin{equation*}
s_{i,j}(\ell)=
\left\{
\begin{array}{lr}
\ell \text{   if } \ell \neq i,j  &  \\
j \text{  if }  \ell = i & \\
i \text{  if }  \ell = j
\end{array}
\right.
\end{equation*}
The map $t_{i,j}:J\minus i \to J\minus j$ is given by:
\begin{equation*}
t_{i,j}(\ell)=
\left\{
\begin{array}{lr}
\ell \text{   if } \ell \neq j  &  \\

i \text{  if }  \ell = j
\end{array}
\right.
\end{equation*}
The map $u_{j}:J \to J\minus j\sqcup *$ is given by:
\begin{equation*}
u_{j}(\ell)=
\left\{
\begin{array}{lr}
\ell \text{   if } \ell \neq j  &  \\

* \text{  if }  \ell = j
\end{array}
\right.
\end{equation*}

We record for future reference the following elementary commutation. For $f:J\to K$ and $i,j \in J$:
\begin{equation}\label{commute}
f^j \circ t_{i,j}=t_{f(i),f(j)} \circ f^i.
\end{equation}
We also remark that these notations can be combined to construct new bijections.  For instance we have a bijection $u_j^i:J\minus i \to J\minus {\{u_j(i),j \}}\sqcup *$, and in particular $u_j^j:J\minus j \to J\minus j$ is the identity map.

\subsection{The category $\R$}
Let $\R_n$ be the category of  covariant functors from $\X_n$ to $\V$, morphisms being natural transformations of functors.   There is an equivalence of $\R_n$ and $Rep(\mathfrak{S}_n)$:
$$\R_n\To Rep(\mathfrak{S}_n)
$$
given by assigning $T$ to $T([n])$, where $[n]=\{1,2,...,n\}$.

Let $\R$ be the category of linear species of \emph{finite degree}, i.e. the category of covariant functors $T:\X\to\V$ such that $T(J)=\{0\}$ for $|J|\gg0$.  There is an equivalence of categories
$$
\R\to\bigoplus_{k=0}^{\infty}Rep(\mathfrak{S}_k)
$$
given by $T\mapsto\bigoplus_{n=0}^{\infty}T([n])$.  (This sum is finite.)

We note $\R$ is a tensor category, with the tensor product defined as follows: for $S,T\in \R$ and $J\in\X$, set $$(S\otimes T) (J)=\bigoplus_{K\subset J} S(K)\otimes T(J\backslash K).$$
The functor $U\in\R$ given by $U(\emptyset)=k$ and $U(J)=0$ for all $J\neq\emptyset$ is the unit of $\R$.

\section{Duality between $\PP$ and $\R$}
\label{SectionS}
Schur-Weyl duality can be reformulated as a duality between the categories $\PP$ and $\R$ as follows.  For $J\in\X$ and $\otimes^J$ denote the $|J|$-fold tensor product. Define the \emph{Schur-Weyl duality functor} $\Schur:\PP\to\R$ by
$$
\Schur(M)(J)=Hom_{\PP}(\otimes^J,M).
$$
Given a morphism $\phi:M\to N$ in $\PP$, we define $\Schur(\phi):\Schur(M)\to\Schur(N)$ in the obvious way: $\Schur(\phi)_J:\Schur(M)(J)\to\Schur(N)(J)$ maps $\psi\in\Schur(M)(J)$ to $\phi \circ \psi \in \Schur(N)(J)$.

One of the goals of this paper is to show that the functor $\Schur$ preserves the higher structures which are inherent to both categories $\PP$ and $\R$. In particular, we will show that it is a morphism of both the $\g$ and $H_\Z$ categorifications structures which we define below.  A more basic property of the Schur-Weyl duality functor is the following:

\begin{proposition}
The functor $\Schur:\PP\to\R$ is monoidal.
\end{proposition}

\begin{proof}
Let $M,N\in\PP$ and $J\in\X$.  We have the following chain of natural isomorphisms:
\begin{align*}
\Schur(M\otimes N)(J)&=Hom_\PP(\otimes^J,M\otimes N)\\
&= Hom_\PP(\otimes^J,\nabla(M\boxtimes N))\\
&\simeq  Hom_{\PP^{[2]}}(\Delta (\otimes^J),M\boxtimes N)\\
&\simeq Hom_{\PP^{[2]}}(\bigoplus_{J_1\cup J_2=J} \otimes^{J_1}\boxtimes \otimes^{J_2}, M\boxtimes N)\\
&=\bigoplus_{J_1\cup J_2=J} Hom_{\PP^{[2]}}(\otimes^{J_1}\boxtimes \otimes^{J_2}, M\boxtimes N)\\
&\simeq \bigoplus_{J_1\cup J_2=J} Hom_{\PP}(\otimes^{J_1},M)\otimes Hom(\otimes^{J_2},N)\\
&=\bigoplus_{J_1\cup J_2=J} \Schur(M)(J_1)\otimes \Schur(N)(J_2)\\
&=(\Schur(M)\otimes\Schur(N))(J)
\end{align*}
Note that we used here that  $\Delta$ and $\nabla$ are adjoint (cf. the proof of Theorem 1.7 in \cite{FFSS} or Lemma 5.8 in \cite{T}).

\end{proof}

\section{The $\g$-categorification}
\label{Sectiong}

In this section we discuss the $\g$-categorification structures on the categories $\PP$ and $\R$, and their relationship via the duality functor.

In Section \ref{Section_Defgcat} we briefly recall the relevant notions, including the definitions of $\g$-categorifications and morphisms of such structures.

In Section \ref{Section_gcatonR} we review the $\g$-categorification on $\R$.  This is a categorification of the basic representation of $\g$, and is well-known, although here it is presented in setting of species.  Weak versions of this statement go back at least to the works of Lascoux, Leclerc, and Thibon \cite{LLT}.  The strong version which incorporates  the theory developed by Chuang and Rouquier, appears already in their  paper \cite{CR}.

In Section \ref{Section_gcatonP} the $\g$-categorification on $\PP$ is recalled \cite{HTY}.

Finally, in Section \ref{Section_Smorphofg}, we prove the following theorem: the Schur-Weyl duality functor $\bold{S}:\PP\to\R$ is a morphism of $\g$-categorifications, which categorifies the standard projection $\pi:B\twoheadrightarrow V(\Lambda_0)$ (cf. Section \ref{MoresymmetriesonB}).

\subsection{Definition of $\g$-categorification}
\label{Section_Defgcat}

To give the definition of $\g$-categorification we use here,  due to Chuang and Rouquier, we first recall the degenerate affine Hecke algebra.

\begin{definition}
Let $DH_n$ be the \emph{degenerate affine Hecke algebra} of $GL_n$.  As an abelian group
$$
DH_n=\Z[y_1,...,y_n]\otimes \Z \mathfrak{S}_n.
$$
The algebra structure is defined as follows: $\Z[y_1,...,y_n]$ and $\Z \mathfrak{S}_n$ are subalgebras, and the following relations hold between the generators of these subalgebras:
$$
\sigma_iy_j=y_j\sigma_i \text{ if }|i-j| \geq 1
$$
and
\begin{equation}
\label{HeckeRelation}
\sigma_iy_{i}-y_{i+1}\sigma_i=1
\end{equation}
(here $\sigma_1,...,\sigma_{n-1}$ are the simple generators of $\Z \mathfrak{S}_n$).
\end{definition}

\begin{definition}
\label{DefinitionCategorification}[Definition 5.29 in \cite{R}]
Let $\mathcal{C}$ be a  $k$-linear abelian category.  A \emph{$\g$-categorification} on $\mathcal{C}$ is the data of:
\begin{enumerate}
\item An adjoint pair $(E,F)$ of exact functors $\mathcal{C} \rightarrow \mathcal{C},$
\item morphisms of functors $X \in End(E)$ and $\sigma \in End(E^2), $ and
\item a decomposition $\mathcal{C}=\bigoplus_{\omega \in P} \mathcal{C}_{\omega}$.
\end{enumerate}
Let $X^{\circ} \in End(F)$ be the endomorphism of $F$ induced by adjunction.  Then given $a \in k$ let $E_a$ (resp. $F_a$) be the generalized $a$-eigensubfunctor of $X$ (resp. $X^{\circ}$) acting on $E$ (resp. $F$).  We assume that
\begin{enumerate}
\item[4.] $E=\bigoplus_{i \in \Z/p\Z} E_i$,
\item[5.]for all $i$, $E_i(\mathcal{C}_{\omega}) \subset \mathcal{C}_{\omega+\alpha_i}$ and $F_i(\mathcal{C}_{\omega}) \subset \mathcal{C}_{\omega-\alpha_i}$,
\item[6.]  the action of $\{ [E_i],[F_i]\}_{i \in \Z/p\Z}$ on $K_{0}(\mathcal{C})=\bigoplus_{\omega\in P} K_0(\C_\omega)$ gives rise to an integrable representation of $\g$,
\item[7.] the functor $F$ is isomorphic to the left adjoint of $E$, and
\item[8.] the degenerate affine Hecke algebra $DH_n$ acts on $End(E^n)$ via
\begin{equation}
\label{EqXX}
y_i \mapsto E^{n-i}XE^{i-1} \text{ for }1 \leq i \leq n,
\end{equation}
and
\begin{equation}
\label{EqTT}
\sigma_i \mapsto E^{n-i-1}\sigma E^{i-1} \text{ for }1 \leq i \leq n-1.
\end{equation}
\end{enumerate}
\end{definition}
\begin{remark}
The notion of $\g'$-categorification can be defined similarly, but without taking into account the compatibility with the  weight decomposition cf. Definition 5.32 in \cite{R}. It is immediate that a $\g$-categorification induces a $\g'$-categorification.
\end{remark}

Now that we introduced the notion of $\g$-categorification, we discuss morphisms of categorifications.

In their study of $\mathfrak{sl}_2$-categorifications, Chuang and Rouquier defined a notion of morphisms in their setting, and in fact the definition can be generalized to $\g$-categorifications.  Below, we adopt their definition of a morphism of categorifications to our setting, and show that it is a reasonable generalization.

\begin{definition}[cf. Section 5.2.1 in \cite{CR}]
\label{DefgMorphism}
Let $\mathcal{C}$ and $\mathcal{C}'$ be $\g$-categorifications with associated data $(E,F, X, \sigma,\mathcal{C}=\oplus_\omega\mathcal{C}_\omega)$ and $(E',F', X', \sigma',\mathcal{C}'=\oplus_\omega\mathcal{C}'_\omega)$.  Let $(\eta,\epsilon)$ (resp. $(\eta',\epsilon')$) be the unit/conuit for the adjunction $(E,F)$ (resp. $(E',F')$). A \emph{morphism of $\g$-categorifications} is the data of an additive  functor
$
\Phi:\mathcal{C} \rightarrow \mathcal{C'}
$
such that
\begin{equation}
\label{Additional_Condition_Morphism}
\Phi(\C_\omega)\subset \C'_\omega,
\end{equation}
along with isomorphisms of functors
\begin{eqnarray*}
\zeta_{+}:\Phi E \rightarrow E'  \Phi  \\ \zeta_{-}:\Phi  F \rightarrow F' \Phi
\end{eqnarray*}
such that the following diagrams commute:

\begin{enumerate}
\item
$$
\xymatrix{
 & \Phi \ar[dl]_{\Phi \eta } \ar[rd]^{\eta' \Phi} & \\
\Phi FE \ar[r]^{\zeta_{-}E} & F'\Phi E \ar[r]^{F' \zeta_{+}} &   F'E'\Phi
}
$$
\item
$$
\xymatrix{
\Phi E \ar[r]^{\zeta_+} \ar[d]_{\Phi X} & E' \Phi \ar[d]^{X'\Phi} \\ \Phi E \ar[r]^{\zeta_+} &E' \Phi
}
$$
\item
$$
\xymatrix{
\Phi EE \ar[r]^{\zeta_+E} \ar[d]_{\Phi \sigma'} & E'\Phi E \ar[r]^{E\zeta_+} & E'E'\Phi \ar[d]^{\sigma \Phi} \\
\Phi EE \ar[r]^{\zeta_+E} & E'\Phi E \ar[r]^{E\zeta_+} & E'E'\Phi}
$$
\end{enumerate}
\end{definition}
\begin{remark}
\label{gprimecateg}
A morphism of $\g'$-categorifications  is the same as a morphism of $\g$-categorifications, except one ignores condition (\ref{Additional_Condition_Morphism}).
\end{remark}

To argue that the above definition of morphism of categorifications is reasonable, we must show that it categorifies  morphisms of $\g$-modules.  With this in mind, suppose $\mathcal{C}$ is a $\g$-categorification, with all the corresponding data as in the above definition.  Then for every $i\in\Z/p\Z$, the functors $E_i$ and $F_i$ are adjoint, $X$ restricts to a morphism in $End(E_i)$, and $T$ restricts to a morphism in $End(E_i^2)$.  Moreover, it follows directly from the definition of $\g$-categorification that this data defines an $\mathfrak{sl}_2$-categorification as in \cite{CR}.  (This is analogous to restricting a $\g$-module to an $\mathfrak{sl}_2$ root subalgebra in $\g$.)  We denote this $\mathfrak{sl}_2$-categorification structure on $\C$ by $\C_i$.

\begin{proposition}\label{Morphism_Categorification}
In the setting of Definition \ref{DefMorphism}, suppose that $(\Phi,\zeta_+,\zeta_-):\mathcal{C}\to\mathcal{C}'$ is a morphism of $\g$-categorifications. Then for every $i\in\Z/p\Z$, $(\Phi,\zeta_+,\zeta_-)$ canonically induces morphisms $(\Phi,\zeta_{+,i},\zeta_{-,i}):\mathcal{C}_{i}\to\mathcal{C}_{i}'$ of $\mathfrak{sl}_2$-categorifications.  \end{proposition}
\begin{proof}
We have to construct $\zeta_{i,+}:\Phi E_i\rightarrow E'_i \Phi $
and $\zeta_{i,-}: \Phi F_i\rightarrow F'_i \Phi$ from $\zeta_+$ and
$\zeta_-$. Let $E_{i,n}$ be the kernel of $(X-i)^n: E\rightarrow E$.
Then $0\subset E_{i,1}\subset E_{i,2}\subset ...$. By hypothesis, for any $A\in \C$, there exists $N\gg 0$, such that $E_i(A)$  is equal to $E_{i,N}(A)=E_{i,N+1}(A)=\cdots$. The analogous statement holds also for $E'$. Then there exists $N'$, such that for any $\ell \geq N'$, $E_i(A)=E_{i,\ell}(A)$ and $E'_{i}(\Phi(A))=E'_{i,\ell}(\Phi(A))$.
 By diagram 2 in Definition \ref{DefgMorphism}, for any $n \geq 0$, the following diagram commutes:
\begin{equation}
\xymatrix{ \Phi E(A) \ar[r]^{\zeta_+} \ar[d]_{\Phi (X-i)^n} & E' \Phi(A) \ar[d]^{(X'-i)^n \Phi} \\
\Phi E(A) \ar[r]^{\zeta_+} & E' \Phi(A)
}
\end{equation}
In particular when $\ell\geq N'$, $E_{i,\ell}(A)$ is a direct summand of $E(A)$ and $E'_{i,\ell}(\Phi(A))$ is a direct summand of $E'(\Phi(A))$. Then $\Phi(E_{i,\ell}(A))$ is also a direct summand of $\Phi(E(A))$, since $\Phi$ is an additive functor. Then the above diagram induces an isomorphism $\zeta_{+,i,\ell}:
\Phi(E_{i,\ell}(A)) \rightarrow E'_{i,\ell}(\Phi(A))$. This isomorphism doesn't depend on $\ell\geq N'$, and it is  functorial in $A$. We thus obtain the isomorphism $\zeta_{+,i}:\Phi E_i\rightarrow E'_i\Phi$.    Similarly, we construct isomorphisms $\zeta_{-,i}: \Phi F_i \rightarrow F'_i \Phi$.  Since $(\Phi,\zeta_+,\zeta_-)$ is a morphism of $\g$-categorifications, it follows formally from the construction that $(\Phi,\zeta_{+,i},\zeta_{-,i})$ are morphisms of $\mathfrak{sl}_2$-categorifications.

\end{proof}

Let $\Phi: \C\to \C'$ be a morphism of $\g$-categorifications as in Definition \ref{DefgMorphism} and assume further that $\Phi$ is left exact. Let $R\Phi$ be the right derived functor of $\Phi$.  If for any $A\in \C$ the right derived functors $R^i\Phi(A)=0$ for $i\gg 0$, then $R\Phi: D^b(\C)\to D^b(\C')$ is well-defined.
\begin{corollary}
\label{Morphism_Representation}
Under the above assumptions, $R\Phi$ yields a $\g$-morphism from $K_{0}(\C)$ to $K_{0}(\C')$. In particular, if $\Phi$ is exact, then $\Phi$ yields a $\g$-morphism from $K_{0}(\C)$ to $K_{0}(\C')$.
\end{corollary}
\begin{proof}
The map $[R\Phi]: K_{0}(\C)\to K_{0}(\C')$ is defined as follows: for any $A\in \C$, $[R\Phi]([A])=\sum_i (-1)^i[R^i\Phi(A)]$. Consider the natural isomorphism $\Phi E_i\simeq E'_i \Phi$ from Proposition \ref{Morphism_Categorification}. It induces a natural isomorphism $R\Phi\circ E_i\simeq E'_i\circ R\Phi$, since $E_i$ and $E'_i$ are exact. The analogous statement holds also for functors $F_i$ and $F_i'$. From the definition of morphism of $\g$-categorifiations, $[R\Phi]$ also preserves weight decompositions. Then the corollary follows.
\end{proof}

\begin{remark}
There is an analogue of Corollary \ref{Morphism_Representation} when $\Phi$ is right exact and the left derived functor $L\Phi$ is defined on $D^b(\C)$.
\end{remark}

\subsection{The $\g$-categorification on $\PP$}
\label{Section_gcatonP}

In \cite{HTY} we defined a $\g$-categorification on $\PP$, categorifying the Fock space representation of $\g$.  In this section we briefly recall this result.

Let $\I:\PP\to\PP$ be the identity functor. Define $\E,\F:\PP\to\PP$ as follows.  For $M\in\PP$, set
$$
\E(M)=\Delta M_{*,1}(\cdot,k)\text{, and}
$$
$$
\F(M)=M\otimes I.
$$
These functors are adjoint.  In fact, they are bi-adjoint, i.e. both $(\E,\F)$ and $(\F,\E)$ are adjoint pairs (Propositions 5.1 and 5.3 in \cite{HTY}).  Although for the purposes of $\g$-categorification we only need to make explicit the first adjunction, later on when we discuss Heisenberg categorification the bi-adjunction data will be crucial.  Therefore we take the opportunity now to present all this data.

For $V\in\V$ and $v\in V$, associate a map $\tilde{v}:V\oplus k\to V$ by $\tilde{v}(w,a)=w+av$.  To $\xi \in V^*$ associate $\tilde{\xi}:V\to V\oplus k$ by $\tilde{\xi}(v)=(v,\xi(v))$.  Note that
$$
\F\E(M)(V)=M(V\oplus k)_1\otimes V
$$
$$
\E\F(M)(V)=M(V\oplus k)_1\otimes V \oplus M(V\oplus k)_0\otimes k
$$
We proceed to define units/counits
$$
\xymatrix{
\I \ar[r]^{\eta_1} & \F\E \ar[r]^{\epsilon_2} & \I \\ \I \ar[r]^{\eta_2} & \E\F \ar[r]^{\epsilon_1} & \I
}
$$
as follows:
\begin{enumerate}
\item[$\eta_1$:]Choose a basis $\{e_i\}$ of V and let $\xi_i$ be the dual basis.  Then for $m\in M(V)$ let  $\eta_{M,V}(m)=\sum_i (M(\tilde{\xi}_i)(m))_1\otimes e_i$.  Here we employ the notation $(\cdot)_1$ to denote the weight one component of an element of $M(V\oplus k)$ relative to the action of $GL(k)$.
\item[$\epsilon_2$:]  For $m\otimes v \in \F\E(M)(V)$ let $(\epsilon_2)_{M,V}(m\otimes v)=M(\tilde{v})(m)$.

\item[$\eta_2$:]  For $m \in M(V)$, let $(\eta_2)_{M,V}(m)=(0,M(i_V)(m)\otimes 1)$.  Here $i_V$ is the natural embedding $V \hookrightarrow V\oplus k$.

\item[$\epsilon_1$:]  For $x=(n\otimes v,m\otimes1)\in \E\F(M)(V)$, let $(\epsilon_1)_{M,V}(x)=M(p_V)(m)$.  Here $p_V$ is the natural projection $V\oplus k \twoheadrightarrow V$.
\end{enumerate}
A routine computation shows that indeed the formulas above define natural transformations of functors.  This datum  can be derived from the abstract adjunctions for $\E$ and $\F$ proved in Propositions 5.1 and 5.3 in \cite{HTY}. The following is a restatement of these propositions.\begin{proposition}
The data $(\E,\F,\eta_1,\epsilon_1)$ and $(\F,\E,\eta_2,\epsilon_2)$ are adjunctions.
\end{proposition}

The next datum we need to introduce are the natural transformations $X$ and $\sigma$ acting on $\E$ and $\E^2$.

For any $V\in \V$, let $U(\mathfrak{gl}(V\oplus k))$ denote the enveloping algebra of $\mathfrak{gl}(V\oplus k)$, and let $X_V \in U(\mathfrak{gl}(V \oplus k))$ the \emph{normalized split Casimir operator}, which is defined as follows.

Fix a basis $V=\bigoplus_{i=1}^n ke_i$; this choice induces a basis of $V \oplus k$.  Let $x_{i,j} \in \mathfrak{gl}(V \oplus k)$ be the operator mapping $e_j$ to $e_i$ and $e_{\ell}$ to zero for all $\ell \neq j$.  Then $$X_V=\sum_{i=1}^{n} x_{n+1,i}x_{i,n+1}-n.$$
The element $X_V$ does not depend on the choice of basis. The universal enveloping algebra $U(\mathfrak{gl}(V\oplus k))$ acts on $M(V\oplus k)$, and we define $X_{M,V}$ to be the action of $X_V$ on $M(V\oplus k)$.

In \cite{HY} we proved that $X_V \in U(\mathfrak{gl}(V\oplus k))^{GL(V)\times GL(k)}$.  Hence $X_{M,V}$ defines an operator on $\E(M)(V)$, and this defines the natural transformation $X:\E\to\E$ (see Section 5.2 of \cite{HTY} for more details on this construction).

We next introduce the natural transformation $\sigma:\E^2 \to \E^2$.
Let $M\in\PP$ and $V\in\V$.  By definition, $$\E^2(M)(V)= M_{*,1,1}(V\oplus k\oplus k)=M(V\oplus k\oplus k)_{1,1},$$ where $M_{*,1,1}(V\oplus k\oplus k)$ denotes the $(1,1)$ weight space of $GL(k)\times GL(K)$ acting on $M(V\oplus k \oplus k)$.  Consider the map $V \oplus k \oplus k \to V \oplus k \oplus k$ given by $(v,a,b)\mapsto(v,b,a)$.  Applying $M$ to this map we obtain a morphism:
$$
\sigma_{M,V}:M_{*,1,1}(V\oplus k\oplus k)\to M_{*,1,1}(V\oplus k\oplus k).
$$
This defines the morphism $\sigma$ acting on $\E^2$.

Finally, we need to define the weight decomposition of $\PP$.  By a theorem of Donkin's, the blocks of $\PP$ are parameterized by pairs $(d,\tilde{\lambda})$, where $d$ is a nonnegative integer and $\tilde{\lambda}$ is a $p$-core of a partition of size $d$ \cite{D}.  Consequently the block decomposition of $\PP$ is  naturally indexed by the weight lattice of $\g$:
$$
\PP=\bigoplus_{\omega \in P}\PP_\omega
$$
See Section 5.3 of \cite{HTY} for further details and references.
We can now state Theorem 6.1 of \cite{HTY}:
\begin{theorem}
\label{HTYthm}
Let $p\neq2$.  The category $\PP$ along with the data of adjoint functors $\E$ and $\F$, operators $X\in End(\E)$ and $\sigma\in End(\E^2)$, and the weight decomposition $\PP=\bigoplus_{\omega\in P}\PP_{\omega}$ defines a $\g$-categorification which categorifies the Fock space representation of $\g$.
\end{theorem}

\subsection{The $\g$-categorification on $\R$}
\label{Section_gcatonR}

One of the first examples of $\g$-categorification is on the direct sum of representations of symmetric groups.  We recall this result now, recast in the language of linear species that's more suitable for our purposes.

We begin by defining functors $\E',\F':\R\to\R$. (All the data defined on the category $\R$ is ``primed'' since we reserve the non-primed notation for $\PP$.)  Let $\I:\R\to\R$ denote the identity functor.

 For any linear species $S\in \R$, set
 $$\E'(S)(J)=S(J\sqcup *).$$ On bijections $f:J\to J'$, $\E'(S)(f)=S(f\sqcup 1_*)$.  For a morphism $\phi:S\to T$ in $\R$, $\E'(\phi)_J:\E'(S)(J)\to\E'(T)(J)$ is given by $\phi_{J\sqcup *}$.

The functor $\F'$ is given by
 $$\F'(S)(J)=\bigoplus_{j\in J} S(J \minus j).$$
On bijections $f:J\to J'$, $\F'(S)(f)=\bigoplus_{j\in J} S(f^j)$.  For a morphism $\phi:S\to T$ in $\R$, $\F'(\phi)_J:F'(S)(J)\to F'(T)(J)$ is given by $\bigoplus_{j\in J}\phi_{J\minus j}$.

We now describe an explicit bi-adjunction between $\E'$ and $\F'$. To begin, let $S\in\R$ and $J\in\X$ and consider $\E'\F'(S)(J)$ and $\F'\E'(S)(J)$:
\begin{eqnarray*}
\E'\F'(S)(J) &=& \bigoplus_{j\in J} S((J\minus j) \sqcup *)\oplus S(J)
\end{eqnarray*}
and
\begin{eqnarray*}
\F'\E'(S)(J) &=& \bigoplus_{j\in J}S((J\minus j) \sqcup *).
\end{eqnarray*}

We construct the following candidate maps for the adjunctions between $\E'$ and $\F'$:
 $$\xymatrix{ \I  \ar[r]^{\eta'_1} & \F'\E' \ar[r]^{\epsilon'_2} & \I}.$$
 $$\xymatrix{ \I  \ar[r]^{\eta_2'} & \E'\F' \ar[r]^{\epsilon'_1} & \I}$$

For $S\in\R$ and $J\in\X$,
\begin{enumerate}
\item[] $(\eta'_1)_{S,J}:S(J)\to \F'\E'(S)(J)$ is given by $m \mapsto (S(u_{j})(m))_{j\in J}$,
\item[] $(\eta'_2)_{S,J}:S(J)\to\E'\F'(S)(J)$ is given by $m \mapsto (0,m)$,
\item[] $(\epsilon'_1)_{S,J}:\E'\F'(S)(J)\to S(J)$ is given by projection onto $S(J)$, and
\item[] $(\epsilon'_2)_{S,J}:\F'\E'(S)(J)\to S(J)$ is given by $(m_j)_{j\in J} \mapsto \sum_{j\in J} S(u_j^{-1})(m_j)$.
\end{enumerate}


\begin{theorem}
The data $(\E',\F',\eta'_1,\epsilon'_1)$ and $(\F',\E', \eta'_2, \epsilon'_2)$ are both adjunctions.
\end{theorem}
\begin{proof}
To see that $(\E',\F',\eta'_1,\epsilon'_1)$ is an adjunction we need to check
\begin{eqnarray}\label{Adj1Eq1}
\epsilon'_1 \E' \circ \E' \eta'_1 &=& 1: \E'\to \E'\F'\E'\to \E' \\ \label{Adj1Eq2}
\F' \epsilon'_1 \circ \eta'_1 \F' &=& 1: \F'\to \F'\E'\F'\to \F'
\end{eqnarray}
When evaluated on $S\in\R$ and $J\in\X$, the left hand side of the first equation is the composition:
\begin{eqnarray*}
S(J\sqcup *) &\longrightarrow& \bigoplus_{j\in J} S((J\minus j)\sqcup * \sqcup *)\oplus S(J\sqcup *) \\ &\longrightarrow& S(J\sqcup *)
\end{eqnarray*}
given by$$
m\mapsto ((S(u_j)(m))_{j \in J},S(u_*)(m))\mapsto S(u_*)(m).
$$
Here $u_j:J\sqcup * \to (J\minus j) \sqcup * \sqcup *$ and $u_*:J\sqcup * \to (J\sqcup *) \minus * \sqcup *=J\sqcup *$.  Since $u_*$ is the identity map this proves the first equation.

Now we prove (\ref{Adj1Eq2}).  When evaluated on $S\in\R$ and $J\in\X$, the left hand side is the composition:
\begin{eqnarray*}
\bigoplus_{i\in J} S(J\minus i) &\longrightarrow& \bigoplus_{i,j\in J} S(J\minus \{i,j\}\sqcup *)\oplus \bigoplus_{j\in J}S(J\minus j) \\ &\longrightarrow& \bigoplus_{i\in J} S(J\minus i)
\end{eqnarray*}
where for $m\in S(J\minus i)$, $$
m\mapsto ((S(u_j^i)(m))_{j \in J})\mapsto S(u_i^i)(m).
$$
Here $u_j^i:J\minus i \longrightarrow J\minus \{j,u_j(i)\} \sqcup * $, and in particular $u_i^i:J\minus i \to J\minus i$ is the identity map.  This proves (\ref{Adj1Eq2}).

The proof of the second adjunction follows from a similar computation, which is left to the reader.
\end{proof}

Next we introduce the \emph{Jucys-Murphy morphism} $X'$ on $\E'$. For $S\in\R,J\in\X$, define $X'_{M,J}:\E'(S)(J)\to\E'(S)(J)$ by
$$
X'_{S,J}=\sum_{j\in J}S(s_{j,*}).
$$
\begin{lemma}
The Jucys-Murphy morphism $X'$ is an endomorphism of $\E'$.

\end{lemma}
\begin{proof}
Let $S\in\R$ and $J,J'\in\X$.  Given a bijection $f: J\to J'$, for any $j\in J$ we have the following commutative diagram:
$$\xymatrix{S(J\sqcup *) \ar[rr]^{S(s_{j,*})} \ar[d]_{S(f\sqcup 1_*)} && S(J\sqcup *) \ar[d]^{S(f\sqcup 1_*)} \\
S(J'\sqcup *) \ar[rr]_{S(s_{f(j),*})} && S(J'\sqcup *)
}
$$
Then summing up $j\in J$, we conclude that $X'_{S,J}$ is functorial in $J$.
It is clear $X'_S$ is functorial in $S$.
\end{proof}

Now note that $(E')^2(S)(J)=S(J\sqcup*\sqcup*)$.  There is a natural bijection
$\sigma_J':J\sqcup*\sqcup* \to J\sqcup*\sqcup*$ which switches the two auxiliary elements and leaves elements in $J$ fixed.  (Since we are using \emph{disjoint} unions here, this bijection is well-defined.)  Define $\sigma_{S,J}'=S(\sigma_J')$.  It is clear that this defines a natural transformation $\sigma':(\E')^2\to (\E')^2$.

Finally, we recall that the block decomposition of $\R$ can also be parameterized by weight of $\g$.  Indeed, by Nakayama's Conjecture \cite{BR}, the blocks of $\R$ (recall that $\R$ is equivalent to $\bigoplus_{d\geq 0} Rep(k\mathfrak{S}_d)$) are  also parameterized by pairs $(d,\tilde{\lambda})$, where $\tilde{\lambda}$ is a $p$-core of a partition of $d$.  Then by the same combinatorial rules as we applied in Section 5.3 of \cite{HTY}, we can associate to such a pair a well-defined weight $\omega\in P$.  In this way we obtain a decomposition
$$
\R=\bigoplus_{\omega\in P}\R_\omega.
$$

A ``weak'' version of the following theorem goes back to the works of Leclerc, Lascoux, and Thibon \cite{LLT}.  The ``strong'' version we state is essentially contained in \cite{CR}, although Chuang and Rouquier don't use the language of linear species and the definition of $\g$-categorification appears in \cite{R} later on.

\begin{theorem}
\label{Oldthm}
The category $\R$ along with the data of adjoint functors $\E'$ and $\F'$, operators $X'\in End(\E')$ and $\sigma'\in End((\E')^2)$, and the weight decomposition $\R=\bigoplus_{\omega\in P}\R_{\omega}$ defines a $\g$-categorification which categorifies the basic representation of $\g$.
\end{theorem}

\subsection{Schur-Weyl duality and the $\g$-categorifications}
\label{Section_Smorphofg}

We have two $\g$-categorifications, $(\PP,\E,\F,X,\sigma)$ and $(\R,\E',\F',X',\sigma')$, of $B$ and $V(\Lambda_0)$, respectively.  In this section we prove our first main theorem which is that the Schur-Weyl duality functor $\bold{S}:\PP\to\R$ can be naturally enriched to a morphism of these $\g$-categorifications, in the sense of Definition \ref{DefMorphism}.  We will see that $\Schur$ categorifies the standard projection $\pi:B\twoheadrightarrow V(\Lambda_0)$(cf. Section \ref{MoresymmetriesonB}).

\subsubsection{Preparatory lemmas}

For $V\in\V$ let $\iota_V: V\to V\oplus k$ be the embedding and $p_V:V\oplus k \to V$ be the projection.
\label{SectionPrep}
\begin{lemma}\label{Polynomial_Functor_Lemma_I}
\label{Restriction Lemma}
Let $V\in\V$ and $M\in\PP$.  Then
\begin{enumerate}
\item $M(\iota_V)$ maps $M(V)$ to $M(V\oplus k)_0$ and $M(\iota_V): M(V)\to M(V\oplus k)_0$ is an isomorphism.
 \item $M(p_V)$ maps $M(V\oplus k)_0$ isomorphically onto $M(V)$ and $M(p_V)(M(V\oplus k)^i)=\{0\}$ for all $i>0$.
\end{enumerate}
\end{lemma}
\begin{proof}
We first prove that $M(\iota_V): M(V)\to M(V\oplus k)_0$ is an isomorphism.

It is easy to check $M(\iota_V)$ maps $M(V)$ to $M(V\oplus k)_0$.
Suppose $M$ is of degree $d$.  The group $GL(V)$ acts on the functor $\Gamma^{V,d}$, and this induces a representation of $GL(V)$ on the vector space $Hom(\Gamma^{V,d},M)$. By Theorem 2.10 in \cite{FS97}, $M(V)$ is canonically isomorphic to $Hom(\Gamma^{V,d},M)$ as $GL(V)$-modules.

We thus need to check that $Hom(\Gamma^{V,d},M)\simeq Hom(\Gamma^{V\oplus k,d},M)_0$. We choose a basis of $V$, such that $V=k^n$. Then
the functor $\Gamma^{{n},d}=\Gamma^{V,d}$ can be decomposed canonically as
$$\Gamma^{{n},d}=\bigoplus_{d_1+d_2\cdots +d_{n}=d} \Gamma^{d_1}\otimes\cdots \otimes \Gamma^{d_{n}}.$$
By Corollary 2.12 in \cite{FS97}, $\Gamma^{d_1}\otimes\cdots \otimes \Gamma^{d_{n}}$  represents the weight space of $M(k^n)$ with weight $(d_1,d_2,...,d_n)$. In other words,
\begin{equation}
\label{EquationFS}
Hom(\Gamma^{d_1}\otimes\cdots \otimes \Gamma^{d_n},M) \simeq M(k^n)_{(d_1,...,d_n)},
\end{equation}
where $M(k^n)_{(d_1,...,d_n)}$ is the weight space corresponding to the character $(d_1,...,d_n)$.
Hence
\begin{eqnarray*}
Hom(\Gamma^{{n},d},M) &\simeq& \bigoplus_{d_1+d_2\cdots +d_{n}=d} Hom(\Gamma^{d_1}\otimes\cdots \otimes \Gamma^{d_{n}},M) \\ &\simeq& Hom(\Gamma^{{n+1},d},M)_0.
\end{eqnarray*}

Secondly, since $M(p_V)\circ M(\iota_V)=id_{M(V)}$, it follows that $M(p_V):M(V\oplus k)_0\to M(V)$ is also an isomorphism.

Finally, by degree considerations,  when $i\geq 1$, $M(p_V)$ maps $M(V\oplus k)_i$ to zero.
\end{proof}

Recall the construction of $\tilde{\xi}:V\to V\oplus k$ from $\xi\in V^*$, which we introduced in Section \ref{Section_gcatonP}: $\tilde{\xi}(v)=(v,\xi(v))$.

\begin{lemma}\label{Xi_Lemma}
For any $\xi\in V^*$, the map $M(\tilde{\xi})_0:M(V)\to M(V\oplus k)_0$ is equal to $M(\iota_V)$, and so it is an isomorphism.
\end{lemma}
\begin{proof}
To show $M(\tilde{\xi})_0=M(\iota_V)$, it is enough to show $M(p_V)\circ M(\tilde{\xi})_0=M(p_V)\circ M(\iota_V)$, since $M(p_V): M(V\oplus k)_0\to M(V)$ is an isomorphism by Lemma \ref{Polynomial_Functor_Lemma_I}. Moreover, also by Lemma \ref{Polynomial_Functor_Lemma_I}, $M(p_V)$ kills all $M(V\oplus k)_j$, if $j>0$. Hence $M(p_V)\circ (M(\tilde{\xi}))_0=M(p_V)\circ M(\tilde{\xi})=M(id_V)=M(p_V)\circ M(\iota_V)$.
\end{proof}

Let $\varpi_n=(1,...,1)$ be the character of $GL_n$ corresponding to the determinant representation. By Equation (\ref{EquationFS}), for any $M\in\PP$ and $n\geq0$,
\begin{equation}
\label{EqOmega}
Hom_{\PP}(\otimes^n,M)\simeq M(k^n)_{\varpi_n}
\end{equation}
There exists a canonical right action of $\mathfrak{S}_n$ on $\otimes^n$, which induces a left action of $\mathfrak{S}_n$ on the  left hand side of (\ref{EqOmega}). Moreover, the group of permutation matrices acts on the right hand side.

\begin{lemma}
\label{Group_Compatibility}
Equation (\ref{EqOmega}) is an isomorphism of $\mathfrak{S}_n$-modules.
\end{lemma}
\begin{proof}
Suppose $M$ is of degree $d$. Theorem 2.10 in \cite{FS97} says that the natural isomorphism
$Hom(\Gamma^{n,d},M)\simeq M(k^{ n})$  is a $GL_n$-isomorphism, and  $\otimes^n$ represents the $\omega_n$ weight space. It is straightforward to see that the right action of $\mathfrak{S}_n$ on $\otimes^n$ coincides with the right action of group of permutation matrices in $GL_n$.
\end{proof}

We take the opportunity now to prove a lemma which will be useful later on.  Recall that we defined an operator $\sigma:\E^2\to\E^2$ in Section \ref{Section_gcatonP}.  There is also a natural operator $\tau:\F^2\to\F^2$ which just flips the two factors, i.e. for $M\in\PP$ and $V\in\V$ then $\F^2(M)(V)=M(V)\otimes V\otimes V$ and $\tau$ maps $m\otimes v \otimes w \mapsto m\otimes w \otimes v$.

\begin{lemma}\label{Permutation_Flip_Com}
The operators $\sigma$ and $\tau$ are compatible with respect to the adjunction $(\F,\E,\eta_2,\epsilon_2).$
\end{lemma}
\begin{proof}
We need to check that given any $M,N\in \PP$  the  isomorphism induced by the adjunction $(\F,\E, \eta_2,\epsilon_2)$
$$Hom(\F^2 M,N)\simeq Hom(M,\E^2N),$$ intertwines the $\mathfrak{S}_2$-action induced by $\tau$ and $\sigma$ respectively. Since any object $M\in \PP$ is a sum of subquotients of functors of the form $\otimes^n$ (see Remark 4.4 in \cite{HTY}),  it is enough to check the lemma for $\otimes^n$, and this follows from Lemma \ref{Group_Compatibility}.
\end{proof}

By Kuhn duality the above lemma implies also the analogous result for other adjunction:

\begin{lemma}\label{Second_Adjunction}
The operators $\sigma$ and $\tau$ are compatible with respect to the adjunction $(\E,\F,\eta_1,\epsilon_1)$.
\end{lemma}

The following lemma will also be used.  

\begin{lemma}
\label{Two_Operators}
Suppose $char(k)\neq2$.  Let $V$ be a polynomial representation of degree 2 of $GL_2$, and let $e$ and $f$ be the standard Chevalley generators of $\mathfrak{gl}_2$. Then on the $\varpi_2$-weight space of $V$ the operator $ef-1$ coincides with the action of
 $$
s= \left(
 \begin{array}{cc}
 0 & 1 \\
 1 & 0 \\
 \end{array}
 \right).
 $$
\end{lemma}
\begin{proof}
 We denote $V_{\omega_2}$ the $\omega_2$-weight space of $V$.   Since $char(k)\neq2$, any polynomial representation of degree $2$ is semisimple. Hence
$$
V_{\omega_2} \subset  V(1,1) \oplus V(2,0),
$$
where $V(i,j)$ is the isotypic component of $V$
corresponding to  the irreducible representation of $GL_2$ of
highest weight $(i,j)$.  Hence any  $v\in V_{\omega_2}$ decomposes as
$v=v'+v''$, where $v'\in V(1,1)$ and $v''\in V(2,0)$. Now one computes separately on $v'$ and $v''$ that the actions of $ef-1$ and $s$ coincide.
\end{proof}

\subsubsection{$\bold{S}$ is a morphism of $\g$-categorifications}
\label{SubSectionS_morphgcat}

We are now ready to enrich the functor $\Schur$ to a morphism of $\g$-categorifications (cf. Definition \ref{DefgMorphism}).
First we introduce the isomorphisms of functors
$
\zeta_+:\bold{S}\E\to\E'\bold{S}
$
and
$
\zeta_-:\bold{S}\F\to\F'\bold{S}.
$

By definition, the transformation $\zeta_+:\Schur\E\to\E'\Schur$ consists of a family of morphisms $(\zeta_+)_{M,J}:(\Schur\E) (M)(J)\to (\E'\Schur) (M)(J)$, where $M\in\PP$ and $J\in\X$.  Of course this family must satisfy the usual naturality conditions.  Now, by our constructions, $$(\Schur\E) (M)(J) =Hom_{\PP}(\otimes^J, \E(M)),$$ and $$(\E'\Schur) (M)(J)=Hom_{\PP}(\F(\otimes^J), M).$$  Therefore we can simply define $(\zeta_+)_{M,J}$ to be the isomorphism coming from the adjunction $(\F,\E)$.  It is straight-forward to check that indeed the family of morphisms $\{(\zeta_+)_{M,J}:M\in\PP,J\in\X\}$ satisfies the naturality conditions with respect to morphisms in $\PP$ and $\X$.

Similarly, we use the adjunction $(\E,\F)$ to define $\zeta_-$.  Let $M\in\PP$ and $J\in\X$.  Then
$$
(\Schur\F)(M)(J)=Hom_\PP(\otimes^J,\F(M)),
$$
while
\begin{eqnarray*}
(\F'\Schur)(M)(J)&=& \bigoplus_{j\in J}\Schur(M)(J \minus j) \\ &=& \bigoplus_{j\in J} Hom_\PP(\otimes^{J\minus j},M) \\ &=& Hom_\PP(\bigoplus_{j\in J}\otimes^{J\minus j},M) \\ &=& Hom_\PP(\E(\otimes^J),M)
\end{eqnarray*}
Therefore the adjunction $(\E,\F)$ provides an isomorphism, $$(\zeta_-)_{M,J}:(\Schur\F)(M)(J)\to (\F'\Schur)(M)(J).$$  The family of such isomorphisms define the natural isomorphism $\zeta_-$.

Now that we've introduced the data $\zeta_+$ and $\zeta_-$, we are ready to prove that $(\Schur,\zeta_+,\zeta_-):\PP\to\R$ satisfies the conditions making it a morphism of $\g$-categorifications.  The following series of lemmas prove these conditions.

\begin{lemma}\label{Condition_I}
The natural transformations $\zeta_+$ and $\zeta_-$ are compatible with respect to  adjunctions $(\F,\E,\eta_2,\epsilon_2) $ and $(\F',\E',\eta'_2,\epsilon'_2)$, i.e. the following diagram commutes:
$$
\xymatrix{
 & \Schur \ar[dl]_{\Schur \eta_2 } \ar[rd]^{\eta'_2 \Schur} & \\
\Schur \E\F \ar[r]^{\zeta_{+} \F } & \E'\Schur \F \ar[r]^{\E' \zeta_{-}} &   \E'\F'\Schur
}
$$
\end{lemma}
\begin{proof}
Evaluating the above diagram on $M\in\PP$ and $J\in\X$ we obtain:
$$
\xymatrix{
& Hom_\PP(\otimes^J, M) \ar[dl] \ar[ddd] \\ Hom(\otimes^J, \E\F(M)) \ar[d] & \\ Hom(\otimes^{J\sqcup *},\F(M)) \ar[dr] & \\ & \bigoplus_{i\in J\sqcup *}Hom(\otimes^{(J\sqcup *)\minus i},M)
}
$$

So we need to check that the above diagram commutes. For any morphism $f:\otimes^{ J}\rightarrow M$, applying $\Schur \eta_2$ and then  $\zeta_{+} \F$, we get $\F(f):\otimes^{J\sqcup *}\rightarrow \F(M)$, which is equal to $f\otimes1$. 
Then applying  $\E' \zeta_{-}$ we obtain the map $\eta_2'\Schur(f)$.

\end{proof}

The next lemma can be proved  from the above lemmas using Lemmas \ref{Duality_Lemma_1} and \ref{Duality_Lemma_2} in the Appendix.

\begin{lemma}
\label{Condition_I'}
The natural transformations
$\zeta_+$ and $\zeta_-$ are compatible with respect to the adjunction $(\E,\F,\eta_1,\epsilon_1)$ and $(\E',\F',\eta'_1,\epsilon'_1)$, i.e. the diagram commutes:
$$
\xymatrix{
 & \Schur \ar[dl]_{\Schur \eta_1 } \ar[rd]^{\eta'_1 \Schur} & \\
\Schur \F\E \ar[r]^{\zeta_{+} \E } & \F'\Schur \E \ar[r]^{\F' \zeta_{-}} &   \F'\E'\Schur
}
$$
\end{lemma}

\begin{lemma}\label{Condition_II}
The natural transformation $\zeta_+$ is compatible with $X$ and $X'$, i.e. the following diagram commutes:
$$\xymatrix{
\Schur \E \ar[r]^{\zeta_+} \ar[d]_{\Schur X}  & \E'\Schur \ar[d]^{X'\Schur}\\
\Schur \E \ar[r]^{\zeta_+}  &    \E'\Schur
}
$$
\end{lemma}
\begin{proof}
Evaluating the above diagram on $M\in\PP$ and $J\in\X$ gives:
$$\xymatrix{Hom(\otimes^{ J}, \E(M) ) \ar[r] \ar[d] & Hom(\otimes^{ J\sqcup *}, M) \ar[d]\\
Hom(\otimes^{ J}, \E(M)) \ar[r] & Hom(\otimes^{J\sqcup *},M)
}
$$
By Lemma \ref{Group_Compatibility}, we are reduced to show that the split Casimir operator $\sum_{i=1}^{n}x_{n+1, i}x_{i, n+1}-n$ coincides with the Jucy-Murphy element $\sum_{i=1}^{n} (i, n+1) $ on the space $M(k^{n+1})_{\varpi_{n+1}}$. It suffices to check that $(i,n+1)$ coincides with $x_{n+1, i}x_{i, n+1}-1$ on the  $M(k^{n+1})_{\varpi_{n+1}}$; this is precisely the content of Lemma \ref{Two_Operators}.
\end{proof}

The next compatibility follows from Lemma \ref{Group_Compatibility}.

\begin{lemma}\label{Condition_III}
The natural transformation $\zeta_+$ is compatible with $\sigma$ and $\sigma'$, i.e.
the following diagram commutates:
$$\xymatrix{
\Schur \E^2 \ar[r]^{\zeta_+\E} \ar[d]^{\Schur \sigma} & \E' \Schur \E \ar[r]^{\E' \zeta_+} & \E'^2\Schur \ar[d]^{\sigma' \Schur}\\
\Schur \E^2 \ar[r]^{\zeta_+\E} & \E'\Schur \E \ar[r]^{\E' \zeta_+}  &\E'^2 \Schur
 }$$
\end{lemma}

\begin{theorem}
\label{Thm_Smorphgcat}
The data $(\Schur,\zeta_+,\zeta_-)$ is a morphism of $\g$-categorifications from $\PP$ to $\R$, categorifying the standard projection $\pi:B \twoheadrightarrow V(\Lambda_0)$.
\end{theorem}

\begin{proof}
By Lemmas \ref{Condition_I'}, \ref{Condition_II}, and  \ref{Condition_III}, to prove that $(\Schur,\zeta_+,\zeta_-)$ is a morphism of $\g$-categorifications, it remains to show that  $\Schur:  \PP_\omega\to \R_\omega$.

For $\lambda \in \wp$, let $L_\lambda$ be the socle of $S_\lambda$.  Up to isomorphism, the polynomials functors $L_\lambda$ are the simple objects in $\PP$ (cf. Theorem 4.8 in \cite{HTY} and references therein). Therefore it suffices to show that for any $L_\lambda \in \PP_\omega$, $\Schur(L_\lambda)\in \R_\omega$.  Now suppose $\lambda$ is a partition of $n$.  Then $S_\lambda$, and hence $L_\lambda$, are homogenous of degree $n$.  Hence by (\ref{EqOmega}), $\Schur(L_\lambda)\cong L_\lambda(k^n)_{\omega_n}$ as $\mathfrak{S}_n$-modules.  Then by \cite[Theorem 6.4b]{G} $\Schur(L_\lambda)=0$ unless $\lambda$ is ``column $p$-restricted'', i.e. if $\lambda=(\lambda_1,\lambda_2,...)$ then $0\leq\lambda_i-\lambda_{i+1}<p$ for all $i$.  In the case that $\lambda$ is column $p$-restricted then $\Schur(L_\lambda)\cong L_\lambda'$, the irreducible $\mathfrak{S}_n$-module indexed by this partition.

Now recall that $\PP_\omega$ (resp. $\R_\omega$) is the block consisting of objects whose composition factors lie $\omega$. (Here weights of $\g$ correspond to certain classes of partitions.  See Section 5.3 of \cite{HTY} for the relevant combinatorics.)  So if $L_\lambda \in \PP_\omega$ then $\lambda \in \omega$, and therefore $\Schur(L_\lambda) \in \R_\omega$. It follows that  $\Schur(\PP_\omega)\subset \R_\omega$. This shows that $\Schur$ is a morphism of $\g$-categorifications.

By Proposition \ref{Morphism_Categorification}, this implies that $[\Schur]$ is a morphism of $\g$-modules.  Moreover, since $\Schur(k)=U$, where $k$ and $U$ are units in tensor categories $\PP$ and $\R$ respectively. Since the basic representation, as a $\g$-module, occurs with multiplicity one in $B$, we conclude that $[\Schur]$ must categorify the standard projection.

\end{proof}

\begin{remark}
Let $Proj \R$ be the subcategory of $\PP$ consisting of projective objects.
Let $\Schur^*: Proj \R\to \PP$ be the adjoint functor of $\Schur$.  By the machinery developed in Proposition \ref{Adjunction_Morphism_Categorification}, it follows that $\Schur^*$ is a morphism of $\g$-categorifications.  $\Schur^*$ categorifies the embedding from $V(\Lambda_0)$ to $B$. For details see \cite{H}.
\end{remark}

\section{Heisenberg categorification}
\label{SectionHCat}
In this section we undertake the study of Heisenberg categorifications on $\PP$.

In Section \ref{SectionDefHCat} we recall Khovanov's category $\HH$ which is used to define a notion of ``strong'' $H_\Z$-categorification.   This allows us to define morphisms of $H_\Z$-categorifications.  We also record a notion of ``weak'' $H_\Z$-categorification.

In Section \ref{SectionHCatonP} we construct a functor $\HH\to End_L(\PP)$, where $End_L(\PP)$ is the category of exact endo-functors on $\PP$ admitting a left adjoint.  This leads to two theorems; the first is a weak $H_\Z$-categorification of the Fock space representation (which holds for all $p\geq0$), and the second is an $H_\Z$-categorification on $\PP$ when $p=0$.

In Section \ref{SectionHCatonR} we briefly recall Khovanov's $H_\Z$-categorification on $\R$ when $p=0$.

We then show in Section \ref{SectionSmorphHCat} that the Schur-Weyl duality functor is an equivalence of $H_\Z$-categorifications.  This is expected, since $\PP$ and $\R$ are naturally equivalent when $p=0$.  Nevertheless, the fact that $\Schur$ preserves all the higher structure endowed on these categories by the $H_\Z$-categorifications is nontrivial.
\subsection{Definition of $H$-categorification}
\label{SectionDefHCat}

In \cite{Kh} Khovanov defines a (conjectural) categorification of an integral form of the Heisenberg algebra $H$.  This is a monoidal category $\HH$  which is additive and linear over a field of characteristic zero, and  there is an injective homomorphism from $H_{\Z}$ to $K_0(\HH)$.  (Recall that $H_\Z$ is the integral form of $H$ defined by generators $e_r,h_r^* $ and relations as in Example \ref{example_H}(2).)  Conjecturally this map is actually an isomorphism.

Let us now recall briefly Khovanov's constructions.  Following his work, we first introduce a preliminary category $\HH'$, and $\HH$ will be defined as the Karoubi envelope of $\HH'$.

The category $\HH'$ has two generating objects, $Q_+$ and $Q_-$.  Any object of $\HH'$ is a sum of products of these two objects.  Such a product is denoted $Q_\varepsilon$, where $\varepsilon$ is a sequence of $+'s$ and $-'s$.  The object indexed by the empty sequence $\emptyset$  serves as a unit $1=Q_\emptyset\in \HH'$.

The space of morphisms $Hom_{\HH'}(Q_\varepsilon,Q_{\varepsilon'})$ is the vector space generated by suitable planar diagram, modulo local relations.  For the graphical calculus see \cite{Kh}; we opt to define the morphisms algebraically.

The identity morphism of an object $Q_\varepsilon$ is denoted by $Id$.   In addition to the identity morphisms, we introduce distinguished morphisms:
\begin{eqnarray*}
&\eta_1& \in Hom_{\HH'}(1,Q_{+-}) , \epsilon_1 \in Hom_{\HH'}(Q_{-+},1), \\ &\eta_2& \in Hom_{\HH'}(1,Q_{-+}) ,  \epsilon_2 \in Hom_{\HH'}(Q_{+-},1), \\ &\tau& \in Hom_{\HH'}(Q_{++},Q_{++}).
\end{eqnarray*}
The space of all morphisms is generated by the identity morphisms along with the distinguished morphisms, such that the following \emph{local relations} hold:
\begin{enumerate}
\item[(R1)] $\epsilon_1\circ \eta_2=Id$
\item[(R2)] $\tau\circ \tau=Id$
\item[(R3)] $ Q_+\tau\ \circ\tau Q_+\circ Q_+\tau=\tau Q_+\circ Q_+\tau\circ\tau Q_+$
\item[(R4)] $\epsilon_1 Q_+\circ Q_- \tau\circ \eta_2Q_+=0$
\end{enumerate}
To describe the final two conditions set:
$$
\xymatrix{\triangle= Q_{+-} \ar[r]^{\eta_2 Q_{+-}} & Q_{-++-} \ar[r]^{Q_-\tau Q_-}&Q_{-++-} \ar[r]^{Q_{-+}\epsilon_2} & Q_{-+}\\
\square = Q_{-+} \ar[r]^{ Q_{-+}\eta_1} & Q_{-++-} \ar[r]^{Q_-\tau Q_-}& Q_{-++-} \ar[r]^{\epsilon_1 Q_{+-}}&Q_{+-}}
$$
Then
\begin{enumerate}
\item[(R5)] $\square \circ \triangle=Id$
\item[(R6)] $\triangle \circ \square=Id-\eta_2\circ \epsilon_1$
\end{enumerate}
The category $\HH=Kar(\HH')$ is the Karoubi envelope of $\HH'$.  Recall that objects of $\HH$ are pairs $(A,e)$, where $A\in\HH'$ and $e\in Hom_{\HH'}(A,A)$ is an idempotent.  The morphisms $f:(A,e)\to (A',e')$ are morphisms $f:A\to A'$ in $\HH'$ such that $f=f\circ e=e' \circ f$.  Notice that $e:(A,e)\to (A,e)$ is the identity morphism of $(A,e) \in \HH$.

Let $+^n$ (resp. $-^n$) denote the sequence of $n$ $+'s$ (resp. $-'s$).  There is an action of $\mathfrak{S}_n$ on $Q_{+^n}$.  Indeed, the simple transposition $\sigma_i$ acts by $Q_{+^{n-i-1}} \tau Q_{+^{i-1}}$.  By (R2) and (R3) this defines a representation of $\mathfrak{S}_n$ on $End_{\HH'}(Q_{+^n})$.  Similarly, there is an action of $\mathfrak{S}_n$ on $Q_{-^n}$.  We define $\sigma \in Hom_{\HH}(Q_{--},Q_{--})$ by the composition
$$
\xymatrix{Q_{--} \ar[rr]^{Q_-\eta_2\eta_1 Q_{-}} && Q_{--++--} \ar[rr]^{Q_{--}\tau Q_{--}}&&Q_{--++--} \ar[rr]^{Q_{-}\epsilon_1\epsilon_2Q_-} && Q_{--}}
$$
It's clear that $\sigma$ satisfies relations analogous to (R2) and (R3), and consequently we use this morphism to define an action of $\mathfrak{S}_n$ on $Q_{-^n}$.

 Now suppose $char(k)=0$. Then there are distinguished objects $S^n_{-}$ and $\Lambda^n_{+}$ in $\HH$ defined as follows.  Consider the symmetrization and anti-symmetrization idempotents:
\begin{eqnarray*}
e(n)&=&\frac{1}{n!}\sum_{\sigma\in \mathfrak{S}_n}\sigma \\ e'(n)&=&\frac{1}{n!}\sum_{\sigma\in \mathfrak{S}_n}sgn(\sigma)\sigma
\end{eqnarray*}
Set
$$
S^n_-=(Q_{-^n},e(n)) \text{  and }\Lambda^n_{+}=(Q_{+^n},e'(n)),
$$ objects in $\HH$.  Define $\gamma:H_{\Z}\to K_0(\HH)$ by $h_n^* \mapsto [S^n_-]$ and $e_n \mapsto [\Lambda^n_{+}]$.

\begin{theorem}[\cite{Kh}, Theorem 1]
The map $\gamma$ is an injective ring homomorphism.
\end{theorem}
\begin{remark}
Conjecture 1 in \cite{Kh} is that $\gamma$ is an isomorphism.
\end{remark}

For a $k$-linear abelian category $\mathcal{V}$, let $End(\mathcal{V})$ denote the category of exact endofunctors on $\mathcal{V}$.  As usual, morphisms are natural transformations of functors.

\begin{definition}
\label{DefHcat}
\begin{enumerate}
\item
Let $p=0$.   An \emph{$H_\Z$-categorification} of an $H_\Z$-module $V$ is  a pair $(\mathcal{V},\Omega)$ of a  $k$-linear abelian category $\mathcal{V}$ and a $k$-linear monoidal functor $\Omega:\HH \to End(\mathcal{V})$, such that the action of $H_\Z$ on $K_0(\mathcal{V})$, defined by $[\Omega]\circ \gamma$, is isomorphic to $V$.
\item A \emph{weak $H_\Z$-categorification} of an $H_\Z$-module $V$ is  an abelian category $\mathcal{V}$, a family of functors $A_n,B_n:\mathcal{V}\to\mathcal{V}$ ($n\geq 0$), and functorial isomorphisms
\begin{enumerate}
\item[] $A_n\circ B_m \cong B_m\circ A_n \oplus B_{m-1}\circ A_{n-1}$
\item[] $A_n \circ A_m\cong A_m\circ A_n$
\item[] $B_n\circ B_m \cong B_m\circ B_n$
\end{enumerate}
such that the map $H_\Z\to End(K_0(\mathcal{V}))$ given by $h_n^* \mapsto [A_n]$ and $e_n \mapsto [B_n]$ is a representation of $H_\Z$ isomorphic to $V$.
\end{enumerate}
\end{definition}

\begin{remark}
Note that the definition of $H_\Z$-categorification only makes sense when we are working over characteristic zero.  This is because only in this case do we know that $H_\Z$ embeds in $K_0(\HH)$.  On the other hand, the notion of weak $H_\Z$-categorification makes sense for all characteristics.
\end{remark}

Now that we've defined the notion of $H_\Z$-categorification, we can formulate a morphism of such categorifications.  Suppose $(\mathcal{V},\Omega)$ is an $H_\Z$-categorification.  The functor $\Omega$ induces the data of endo-functors $F=\Omega(Q_+),E=\Omega(Q_-)$,  unit/counits coming from the morphisms $\Omega(\epsilon_i),\Omega(\eta_i)$ ($i=1,2$), and the morphism $\Omega(\tau)$.

Conversely, given such data we can reconstruct the functor $\Omega$ by first constructing $\Omega':\HH'\to End(\mathcal{V})$, and then extending to the Karoubi envelope (which is possible since we are assuming that the functor $\Omega$ already exists).  Therefore an $H_\Z$-categorification can be alternatively presented as datum $(\mathcal{V},E,F,\epsilon_1,\eta_1,\epsilon_2,\eta_2,\tau)$ subject to the local relations, and in addition subject to the assumption that the functor $\Omega'$ constructed from this datum extends to the Karoubi envelope.  We write $(\mathcal{V},\Omega)=(\mathcal{V},E,F,\epsilon_1,\eta_1,\epsilon_2,\eta_2,\tau)$ to express this alternative formulation.
\begin{definition}
\label{DefMorphism}
Let $$(\mathcal{V},\Omega)=(\mathcal{V},E,F,\epsilon_1,\eta_1,\epsilon_2,\eta_2,\tau)$$ and  $$(\mathcal{V}',\Omega')=(\mathcal{V}',E',F',\epsilon_1',\eta_1',\epsilon_2',\eta_2',\tau')$$ be $H_\Z$-categorifications.  A \emph{morphism of $H_\Z$-categorifications} is the data of a functor
$$
\Phi:\mathcal{V} \rightarrow \mathcal{V'}
$$
along with isomorphisms of functors
\begin{eqnarray*}
\zeta_{+}:\Phi E \rightarrow E'  \Phi  \\ \zeta_{-}:\Phi  F \rightarrow F' \Phi
\end{eqnarray*}
such that that the following diagrams commute:

\begin{enumerate}

\item
$$
\xymatrix{
 & \Phi \ar[dl]_{\Phi \eta_1 } \ar[rd]^{\eta_1' \Phi} & \\
\Phi FE \ar[r]^{\zeta_{-}E} & F'\Phi E \ar[r]^{F' \zeta_{+}} &   F'E'\Phi
}
$$

\item

$$
\xymatrix{
 & \Phi \ar[dl]_{\Phi \eta_2 } \ar[rd]^{\eta_2' \Phi} & \\
\Phi EF \ar[r]^{\zeta_{+}F} & E'\Phi F \ar[r]^{E' \zeta_{-}} &   E'F'\Phi
}
$$

\item
$$
\xymatrix{
\Phi FF \ar[r]^{\zeta_-F} \ar[d]_{\Phi \tau} & F'\Phi F \ar[r]^{F\zeta_-} & F'F'\Phi \ar[d]^{\tau'\Phi} \\
\Phi FF \ar[r]^{\zeta_-F} & F'\Phi F \ar[r]^{F\zeta_-} & F'F'\Phi}
$$
\end{enumerate}
\end{definition}

The following proposition shows that a morphism of $H_\Z$-categorifications really categorifies a morphism of $H_\Z$-modules. For convenience of exposition, we assume the functor $\Phi$ in the following is exact.  The analogue of this proposition where $\Phi$ is right or left exact can also be proven once we pass to the derived functor.
\begin{proposition}
\label{PropHmorph}
Let $\Phi: (\mathcal{V},\Omega)\to (\mathcal{V}',\Omega')$ be a morphism of $H_\Z$-categorifications. Assume $\Phi$ is exact. Then $[\Phi]:[\mathcal{V}]\to [\mathcal{V}']$ is an $H_\Z$-morphism.
\end{proposition}
\begin{proof}
From the data $\zeta_-$, for every $n$ we can produce the natural isomorphism $$\Phi F^n\simeq F^n \Phi. $$
By Condition 3 in Definition \ref{DefMorphism}, this isomorphism intertwines the action of symmetric group $\mathfrak{S}_n$. Therefore it induces a natural isomorphism $$\Phi\circ \Omega(\Lambda_+ ^n)\simeq \Omega'(\Lambda_+ ^n)\circ \Phi $$
By Conditions 2 and 3, and Lemma \ref{Abstract_Nonsense}, we  get the following commuting diagram
$$\xymatrix{
\Phi EE \ar[r]^{\zeta_+E} \ar[d]_{\Phi \tau^{\vee }} & E'\Phi E \ar[r]^{E'\zeta_+} & E'E'\Phi \ar[d]^{\tau'^{\vee}\Phi} \\
\Phi EE \ar[r]^{\zeta_+E} & E'\Phi E \ar[r]^{E'\zeta_+} & E'E'\Phi},
$$
where $\tau^\vee$ and $\tau'^\vee$ are the  operators on $E^2$ and $E'^2$ induced from $\tau$ and $\tau'$ by appropriate adjunctions.  Then as above, we get a natural isomorphism
$$\Phi \circ \Omega(S_-^n)\simeq \Omega'(S_-^n)\circ \Phi .$$
This shows that $[\Phi]$ intertwines the $H_\Z$-action.

\end{proof}

\subsection{The $H_\Z$-categorification on $\PP$}
\label{SectionHCatonP}

In this section we define an $H_\Z$-categorification in the case when $p=0$, and a weak $H_\Z$-categorification in the case when $p>0$.  In both cases these categorify the Fock space representation of $H_\Z$.

Let $End_L(\PP)$ be the category of endo-functors of $\PP$ that admit a left adjoint.  That is, the objects of $End_L(\PP)$ are exact functors from $\PP$ to $\PP$ that admit a left adjoint, and morphisms are natural transformations of functors.

We now define a $k$-linear monoidal functor  $\Omega':\HH' \to End_L(\PP)$. On objects $\Omega'$ is determined by $\Omega'(Q_+)=\F$ and $\Omega'(Q_-)=\E$.  The distinguished morphisms $\eta_1,\eta_2,\epsilon_1,\epsilon_2$ of $\HH'$ correspond under $\Omega'$ to the eponymous morphisms in $End_L(\PP)$ (defined in Section \ref{Section_gcatonP}).  The morphism $\tau \in \HH'$ corresponds to the  morphism $\tau:\F^2\to\F^2$  (defined in Section \ref{SectionPrep}).

\begin{proposition}
\label{PropWellDef}
The functor $\Omega':\HH'\to End_L(\PP)$ is well-defined.
\end{proposition}

\begin{proof}  We need to check that the local relations (R1)-(R6) are satisfied.  Relations (R1)-(R3) are obvious.

Fix $M\in\PP$ and $V\in\V$.  To check relation (R4), first we note that $\E\F\F(M)(V)$ is equal to
$$
(M(V\oplus k)_1\otimes V\otimes V)\oplus(M(V\oplus k)_0\otimes k\otimes V)\oplus (M(V\oplus k)_0\otimes V\otimes k).
$$
Let $m\otimes v \in M(V)\otimes V$.  Then
$$
(\eta_2\F)_{M,V}(m\otimes v)=M(i_V)(m)\otimes v \otimes 1,
$$
which is in the component $M(V\oplus k)_0 \otimes V\otimes k$ of $\E\F\F(M)(V)$.  Applying $(\E\tau)_{M,V}$, we obtain
 $$
 M(i_V)(m)\otimes 1 \otimes v \in M(V\oplus k)_0\otimes k\otimes V.
 $$
 Now $(\epsilon_1\F)_{M,V}$ is the projection from $\E\F\F(M)(V)$ onto the component $M(V\oplus k)_0\otimes V\otimes k$, followed by the map $M(p_V)\otimes1_V$.  In particular, $M(i_V)(m)\otimes 1 \otimes v$ is mapped to zero, proving (R4).

In order to check (R5) and (R6), let us first note that
\begin{eqnarray*}
\F\E(M)(V) &=& M(V\oplus k)_1\otimes V, \\
\E\F(M)(V)&=&(M(V\oplus k)_1\otimes V) \oplus (M(V\oplus k)_0\otimes k),
\end{eqnarray*}
and $\E\F^2\E(M)(V)$ is equal to
$$
(M(V\oplus k \oplus k)_{1,1}\otimes V \otimes V)\oplus(M(V\oplus k \oplus k)_{0,1}\otimes k \otimes V)\oplus(M(V\oplus k \oplus k)_{0,1}\otimes V \otimes k).
$$

Now we compute $\triangle$.  Let $m\otimes v\in M(V\oplus k)_1\otimes V$.  The operator $(\eta_2\F\E)_{M,V}$ maps $m\otimes v$ to $M(i_{V\oplus k})(m)\otimes v\otimes1$, which is an element of $M(V\oplus k \oplus k)_{0,1}\otimes V \otimes k$.  Next, $(\E\tau\E)_{M,V}$ maps $M(i_{V\oplus k})(m)\otimes v\otimes1$ to $M(i_{V\oplus k})(m)\otimes 1\otimes v$.  Finally, $(\E\F\epsilon_2)_{M,V}$ maps $M(i_{V\oplus k})(m)\otimes 1\otimes v$ to $M(\tilde{1})M(i_{V\oplus k})(m)\otimes v$, where here we write $1$ for the vector $(0,1)\in V\oplus k$.  Now, $M(\tilde{1})M(i_{V\oplus k})=M(\tilde{1}\circ i_{V\oplus k})=M(1_{V\oplus k})$, and therefore $M(\tilde{1})M(i_{V\oplus k})(m)\otimes v=m\otimes v$.  The upshot is that
 $$
 \triangle_{M,V}:\F\E(M)(V)\to\E\F(M)(V)
 $$
is the natural inclusion of $M(V\oplus k)_1\otimes V$ into $(M(V\oplus k)_1\otimes V) \oplus (M(V\oplus k)_0\otimes k)$.

A similar computation shows that
$$
 \square_{M,V}:\E\F(M)(V)\to\F\E(M)(V)
 $$
 is the projection of $(M(V\oplus k)_1\otimes V) \oplus (M(V\oplus k)_0\otimes k)$ onto $M(V\oplus k)_1\otimes V$.

Relation (R5) is immediate from these computations.  To see (R6), note that $M(i_Vp_V)$ restricted to $M(V\oplus k)_0$ is the identity operator.  (This is clearly true for $M=\otimes^d$, and hence for all sums of subfunctors of such functors.)  Therefore, $(\eta_2\circ \epsilon_1)_{M,V}:\E\F(M)(V)\to\E\F(M)(V)$ maps
$$
(y,y')\in \E\F(M)(V)=(M(V\oplus k)_1\otimes V) \oplus (M(V\oplus k)_0\otimes k)
$$
to $(0,y')$.  Relation (R6) follows.
\end{proof}

In order to conclude that $\Omega'$ extends canonically to a functor $\Omega:\HH\to End_L(\PP)$, we need to show that $End_L(\PP)$ is idempotent complete.  Let $\PP^{[2]'}$ denote the category of bi-polynomial functors, which are \emph{contra-variant} in the first variable and \emph{co-variant} in the second variable.  This is an abelian category, and hence idempotent complete. {\color{magenta} }

\begin{proposition}
\label{PropIdem}
The category $End_L(\PP)$ is equivalent to $\PP^{[2]'}$.  In particular, $End_L(\PP)$ is idempotent complete and  $\Omega'$ extends canonically to a functor $\Omega:\HH\to End_L(\PP)$.
\end{proposition}
\begin{proof}
First, we construct the functors between the two categories. Given $T\in End_L(\PP)$ with left adjoint ${}^*T$, construct an associated bi-polynomial functor, $$B_T(V,W)=\bigoplus_{d}{^*T}(\Gamma^{V,d})(W).$$

Conversely, given a bi-polynomial functor $B\in \PP^{[2]'}$, we associate the endo-functor $T_B$ of $\PP$ which is defined by:  $$T_B(M)(V)=Hom_\PP(B(V,*),M(*)).$$
We first show that
$T_B\in End_L(\PP)$ admits the left adjoint $^*T_B$ given by the formula:
$$^*T_B(M)(V)= Hom(M,B(\cdot,V)^*)^* .$$

Recall that we defined $S^{W,d}\in\PP$ by $S^{W,d}(V)=S^d(Hom(W,V))$. Then one can check that $S^{W,d}=(\Gamma^{W^*,d})^\sharp$.  By duality, these are injective co-generators of $\PP$ \cite[Theorem 2.10]{FS97}.  They satisfy the property that for any $M\in\PP$,
$$Hom(M,S^{W,d})\cong M_d(W)^*.$$
Then we have
\begin{align*}
Hom(^*T_B(M),S^{W,d}) &\cong (^*T_B(M)_d)(W)^*\\
                      &=Hom(M,B(\cdot,W)^*_d)\\
                      &=Hom(M, T_B(S^{W,d}))
\end{align*}
Since functors of the form $S^{W,d}$ are injective co-generators in $\PP$,   we have the functorial isomorphism:
$$Hom(^*T_B(M),N)\simeq Hom(M,T_B(N)). $$
This shows that $^*T_B$ is left adjoint to $T_B$, and hence $T_B\in End_L(\PP)$.

We now show that the two functors $T\mapsto B_T$ and $B\mapsto T_B$ are inverse to each other, and therefore define an equivalence $End_L(\PP)\cong \PP^{[2]'}$.

Given $T:\PP\to \PP$ with left adjoint,  consider the new functor $T_{B_T}$.  Then for $M\in\PP$ and $V\in\V$:
\begin{align*}
T_{B_T}(M)(V)&=Hom(B_T(V,\cdot),M)\\
             &=Hom(\bigoplus_d {}^*T(\Gamma^{V,d}),M)\\
             &=\bigoplus_d Hom(^*T(\Gamma^{V,d}),M)\\
             &\cong\bigoplus_d Hom(\Gamma^{V,d},T(M))\\
             &\cong\bigoplus_d T(M)(V)_d\\
             &=T(M)(V)
\end{align*}
Therefore we obtain the functorial isomorphism $T_{B_T}\cong T$.

Conversely, given $B\in\PP^{[2]'}$,  consider the new bi-polynomial functor $B_{T_B}$. Then,
\begin{align*}
Hom(B_{T_B}(V,\cdot),M)&=Hom(\bigoplus_d {}^*T_B(\Gamma^{V,d}),M)\\
                       &=\bigoplus_{d}Hom(^*T_B(\Gamma^{V,d}),M)\\
                       &\simeq\bigoplus_{d} Hom(\Gamma^{V,d},T_B(M))\\
                       &\simeq T_B(M)(V)\\
                       &=Hom(B(V,\cdot),M)
\end{align*}
By the Yoneda Lemma $B_{T_B}\simeq B$.

\end{proof}

By the above proposition, and the universal property of the Karoubi envelope, we have a functor $\Omega:\HH\to End_L(\PP)$.  We will show that when $p=0$ this defines an $H_\Z$-categorification in the sense of Definition \ref{DefHcat}.  First, we introduce a weak $H_\Z$-categorification, which is valid for all characteristics.

Define functors $\bold{A}_r,\bold{B}_r :\PP\to\PP$ as follows: for $M\in \PP$ and $V\in \V$,
\begin{eqnarray*}
\bold{A}_r(M)(V)&=&M(V\oplus k)_r \\ \bold{B}_r(M)(V)&=&M(V)\otimes \Lambda^r.
\end{eqnarray*}
For the definition of $M(V\oplus k)_r$ see Section \ref{SectionBi}.

\begin{theorem}
\label{Thm_Hweakcat}
Let $p\geq0$.  The family of functors $\{\bold{A}_n,\bold{B}_n:n\geq1\}$ defines a weak $H_\Z$-categorification on $\PP$ in the sense of Definition \ref{DefHcat}(2), which categorifies the Fock space representation of $H_\Z$.
\end{theorem}
\begin{proof}
For $M\in\PP$ and $V\in\V$:
\begin{eqnarray*}
\bold{A}_n\circ\bold{B}_m(M)(V) &=& \bold{B}_m(M)(V\oplus k)_n \\ &= & [M(V\oplus k)\otimes\Lambda^m(V\oplus k)]_n \\&=& M(V\oplus k)_n\otimes\Lambda^m(V) \oplus M(V\oplus k)_{n-1}\otimes\Lambda^{m-1}(V)
\end{eqnarray*}
and
\begin{eqnarray*}
\bold{B}_m\circ \bold{A}_n(M)(V) &=& \bold{A}_n(M)(V)\otimes\Lambda^m(V) \\ &=& M(V\oplus k)_n\otimes\Lambda^m(V)
\end{eqnarray*}
Immediately we obtain the  isomorphism,
$$\bold{A}_n\circ \bold{B}_m\cong \bold{B}_m \circ \bold{A}_n \oplus \bold{B}_{m-1}\circ \bold{A}_{n-1}.$$Therefore the family of functors $\{\bold{A}_n,\bold{B}_n:n\geq1 \}$ is a weak $H_\Z$-categorification.

It remains to show that this categorifies the Fock space representation of $H_\Z$.  First note that under $\varrho$, $[\Lambda^n]\mapsto e_n$.  This implies that under this identification $[\bold{B}_n]=e_n$, where here $e_n$ is viewed as the operator $e_n:B\to B$.  Next, note that $[\Gamma^n]\mapsto h_n$  Since $\Gamma^n$ is projective (\cite[Theorem 2.10]{FS97}), by Lemma \ref{LemmaAdjoint} we have that $[\bold{T}_{\Gamma^n}^*]=h_n^*$.  But $\bold{T}_{\Gamma^n}^*=\bold{A}_n$ (\cite[Corollary 2.12]{FS97}), showing that indeed $\{\bold{A}_n,\bold{B}_n:n\geq1 \}$ categorifies the Fock space representation of $H_\Z$.
\end{proof}

\begin{theorem}
\label{Thm_HcatonP}
Let $p=0$ .  Then the functor $\Omega:\HH \to End_L(\PP)$ categorifies the Fock space representation of $H_\Z$ in the sense of Definition \ref{DefHcat}(1). In particular $\Omega(S_-^r)\cong \bold{A}_r$ and $\Omega(\Lambda_+^r)\cong \bold{B}_r$.
\end{theorem}

\begin{proof}
It remains only to show that $[\Omega]\circ \gamma:H_\Z\to End(K_0(\PP))$ is isomorphic to the Fock space representation of $H_\Z$.  For this we will show that $\Omega (S_-^r)\cong \bold{A}_r$ and    $\Omega (\Lambda_+^r)\cong \bold{B}_r$; by the above theorem it will follow that $[\Omega]\circ\gamma$ is isomorphic to Fock space.

By  the construction of $\Omega'$, for  $M\in \PP$ and $V\in \V$,
\begin{eqnarray*}
\Omega(S_-^r)(M)(V)&=&e(r)\cdot \E^r(M)(V)\\&=& e(r)\cdot M(V\oplus k^r)_{\varpi_r},
\end{eqnarray*}
Note that  $\E^r$ is right adjoint to $\F^r$, and by Lemma \ref{Permutation_Flip_Com}, the action of $\mathfrak{S}_r$ on these functors is compatible via the adjunction $(\F,\E)$. It  follows that $\Omega(S_-^r)$ is right adjoint to $\bold{T}_{S^r}$. On the other hand, it is known that the functor $\bold{A}_r$ is right adjoint to the functor $T_{\Gamma^r}$. Since $S^r\cong \Gamma^r$ in characteristic zero,  by the Yoneda lemma we conclude $\Omega(S_-^r)\cong \bold{A}_r$.
It is immediate from the definition that $\Omega(\Lambda _+^r) \cong\bold{B}_r.$

\end{proof}

\subsection{Khovanov's $H$-categorification on $\R$}
\label{SectionHCatonR}

Throughout this section $p=0$.  In \cite{Kh}, Khovanov defines a functor $$\Phi:\HH\to End(\R),$$ which is a categorification of the Fock space representation of $H_\Z$, although he doesn't describe this categorification in terms of linear species.  We briefly recall Khovanov's functor, reformulated  in the setting of linear species.

In fact, we've already introduced most of the data required to define $\Phi$.  In Section \ref{Section_gcatonR} we defined functors $\E',\F'$ and unit/counit data expressing the bi-adjointness of these functors, which we denoted $\epsilon_1',\epsilon_2',\eta_1',\eta_2'$.  Then we define first $\Phi':\HH'\to End(\R)$ by $Q_+\mapsto\F'$ and $Q_-\mapsto\E'$.  We require $\Phi'$ to be a $k$-linear monoidal functor, and hence this determines $\Phi'$ on all objects of $\HH'$.

On morphisms, $\epsilon_i,\eta_i\mapsto\epsilon_i',\eta_i'$ for $i=1,2$.  It remains only to define the image $\tau'$ of the morphism $\tau\in Hom_{\HH'}(Q_{++},Q_{++})$.  For any $S\in \R$ and $J\in \X$,
$$(\F^{'})^2(S)(J)=\bigoplus_{j,\ell \in J} S(J\minus \{j,\ell\}) ^{\oplus 2}.$$
Let $\tau'_{S,J}$ be the map induced by the permutation on $S(J\minus\{j,\ell\}) ^{\oplus 2}$, for any pair $j,\ell \in J$.
This defines an operator $\tau'$ on $\F'^2$.

The following is a reformulation of a theorem of Khovanov's.
\begin{theorem}
[Khovanov]
\label{Thm_HcatonR}
The functor $\Phi':\HH'\to End(\R)$ is well-defined, and extends to a functor $\Phi:\HH\to End(\R)$ which is a categorification of the Fock space representation of $H_\Z$ in the sense of Definition \ref{DefHcat}(1). \end{theorem}




\subsection{$\Schur$ is a morphism of $H_\Z$-categorifications}
\label{SectionSmorphHCat}

We continue to assume that $p=0$.  We've introduced two $H_\Z$-categorifications of the Fock space representation of $H_\Z$, one via the category $\PP$ and the other via the category $\R$.  We now enrich the Schur-Weyl duality functor to an equivalence of these $H_\Z$-categorifications.

Let $(\PP,\Omega)$ be the $H_\Z$-categorification appearing in Theorem  \ref{Thm_HcatonP}.  Let $(\R,\Phi)$ be the $H_\Z$-categorification appearing in Theorem \ref{Thm_HcatonR}.  Let $(\Schur,\zeta_+,\zeta_-)$ be the data of Theorem \ref{Thm_Smorphgcat}.  We claim that this same data also defines a morphism of $H_\Z$-categorifications.

\begin{lemma}
\label{Adjunction_Flip_Prime}
The operators $\sigma'$ on $(\E')^2$ and $\tau'$ on $\F'^2$ are compatible with respect to the adjunction $(\E'^2, \F'^2)$.
\end{lemma}
\begin{proof}
Let $S,T\in\R$. The adjunction $(\E'^2, \F'^2)$ induces the following canonical isomorphism: $$\theta: Hom_\R(\E'^2(S),T)\cong Hom_\R(S,\F'^2(T)).$$
Let us describe $\theta$ explicitly.  Let $f: \E'^2(S)\to T$ be a morphism. Then for any $J\in \X$ we have collection of compatible maps $f_J:S(J\sqcup * \sqcup *)\to T(J)$. Now, if we fix a pair of elements $j,\ell \in J$, then $f_{J\minus \{j,\ell\} }:S(J\minus \{j,\ell\}\sqcup *\sqcup * )\to T(J\minus \{j,\ell\})$. The two bijections from $\{j,l\}$ to $*\sqcup*$ induce two bijections from $J$ to $J\minus \{j,\ell\}\sqcup *\sqcup *$. Hence we get two linear maps from $S(J)$ to $S(J\minus \{j,\ell\}\sqcup *\sqcup * )$. Composing with $f_{J\minus \{j,\ell\}} $ we get two  maps from $S(J)$ to $T(J\minus \{j,\ell\})$. This defines the map $$\theta_J(f): S(J)\to \F'^2(T)(J).$$
Using this description of  $\theta$ the lemma can be easily checked.

\end{proof}

\begin{theorem}
\label{Thm_SequivHcat}
The data $(\Schur,\zeta_+,\zeta_-)$ is an equivalence of $H_\Z$-categorifications in the sense of Definition \ref{DefMorphism}.
\end{theorem}

\begin{proof}
We must show that Diagrams 1-3 of Definition \ref{DefMorphism} commute. Diagram 1 commutes by Lemma \ref{Condition_I'}, and Diagram 2 commutes by Lemma \ref{Condition_I}.  It remains to show that Diagram 3 commutes.

From the adjunction $(\F,\E,\eta_1,\epsilon_1)$ we produce an adjunction $(\F^2,\E^2, \tilde{\eta}_1,\tilde{\epsilon}_1)$, where $\tilde{\eta}_1= \F\eta_1 \E \circ \eta_1$ and $\tilde{\epsilon}_1= \F\epsilon_1 \E \circ \epsilon_1$.  Similarly set $\tilde{\eta}'_1= \F'\eta'_1 \E' \circ \eta'_1$ and $\tilde{\epsilon'}_1= \F'\epsilon'_1 \E' \circ \epsilon'_1$.  This defines an adjunction $(\F'^2,\E'^2, \tilde{\eta}'_1,\tilde{\epsilon}'_1)$.
Let $\tilde{\zeta}_+=\E'\zeta_+\circ \zeta_+\E$ and $\tilde{\zeta}_-=\F'\zeta_-\circ \zeta_-\F$.

The commutativity of Diagram 1 implies that the following diagram commutes:

\begin{equation*}
\xymatrix{&  \Schur \ar[dl]_{\Schur \tilde{\eta}_1} \ar[rd]^{\tilde{\eta'}_1\Schur } \\
\Schur \F^2\E^2 \ar[r]_{\tilde{\zeta_-}\E^2} & \F'^2\Schur \E^2 \ar[r]_{\F'^2\tilde{\zeta_+}} & \F'^2\E'^2 \Schur
.}
\end{equation*}
By Lemma \ref{Condition_III} the following diagram commutates:
\begin{equation*}
\xymatrix{
\Schur \E^2 \ar[r]^{\tilde{\zeta_+}} \ar[d]^{\Schur \sigma} & \E'^2 \ar[d]^{\sigma'\Schur} \Schur \\
\Schur \E^2 \ar[r]_{\tilde{\zeta_+}} &  \E'^2 \Schur }
\end{equation*}
By Lemma \ref{Permutation_Flip_Com}, $\sigma$ and $\tau$ are compatible with respect to the adjunction $(\E^2,\F^2, \tilde{\eta}_1,\tilde{\epsilon}_1)$. By Lemma \ref{Adjunction_Flip_Prime}, $\sigma'$ and $\tau' $ are compatible with respect to the adjunction $(\E'^2,\F'^2, \tilde{\eta}'_1,\tilde{\epsilon}'_1)$.  Therefore we are in the setting of Lemma \ref{Abstract_Nonsense}, and so we  conclude that Diagram 3 commutes:

\begin{equation*}
\xymatrix{\Schur \F^2 \ar[r]^{\tilde{\zeta_-}} \ar[d]_{\Schur \tau} & \F'^2 \ar[d]^{\tau'\Schur} \Schur \\
\Schur \F^2 \ar[r]^{\tilde{\zeta_-}} & \F'^2 \Schur }
\end{equation*}
\end{proof}

\section{$\widehat{\mathfrak{gl}}_p$-categorification}

In this section we define an action of $H^{(1)}$ on $D^b(\PP)$, the bounded derived category of $\PP$, which categorifies the $H^{(1)}$ action on $B$. We also prove the commutativity between the action of $\g'$-action on $\PP$ and the $H^{(1)}$-action on $D^b(\PP)$. Hence it categorifies the action of $U(\g')\otimes H^{(1)}$ on $B.$

 Note that the $U(\g')\otimes H^{(1)}$-action on $B$ is essentially equivalent to the $U(\widehat{\mathfrak{gl}}_p)$-action on $B$; more precisely, as subalgebras of $End(B)$ they are equal.  Therefore, combining the results of this section with the $\g$-action on $\PP$ from Section \ref{Section_gcatonP}, we obtain a (weak) $\widehat{\mathfrak{gl}}_p$-categorification of the Fock space representation (cf. Remark \ref{Remarkglhat}).

\label{SectionTwist}

\subsection{Commutativity between two categorifications}

Recall that we defined $H^{(1)}$ to be the subalgebra of $End(B)$ generated by twisted operators $b^{(1)}$ and their adjoints $(b^{(1)})^*$, for all $b\in B$.  The key combinatorial lemma (Lemma \ref{Lemma_CommutingH}) that we cited was that such operators in fact lie in $End_{\g'}(B)$.  Our first order of business is to categorify this fact.

Suppose $\varrho([M])=b$.  Then this implies $\varrho([M^{(1)}])=b^{(1)}$, and therefore the operator $\bold{T}_{M^{(1)}}$ categorifies the operator $b^{(1)}$ (Lemma \ref{LemmaAdjoint}).  Consequently, to categorify the fact that $b^{(1)}\in End_{\g'}(B)$ we must show that $\bold{T}_{M^{(1)}}$ is a morphism of $\g'$-categorifications.  In order to properly formulate this we first introduce the relevant data.

First we construct $\zeta_+: \bold{T}_{M^{(1)}}\circ \E \cong \E \circ \bold{T}_{M^{(1)}}$ as follows: for any $N\in \PP$ and $V\in\V$ we compute that
$$(\bold{T}_{M^{(1)}}\circ \E (N))(V)=N(V\oplus k)_1\otimes M^{(1)}(V) $$
$$(\E\circ \bold{T}_{M^{(1)}})(N)(V)=(N(V\oplus k)\otimes M^{(1)}(V\oplus k))_1$$
The $GL(k)$-weights of $M^{(1)}(V\oplus k)$ are multiples of $p$, so by Lemma \ref{Polynomial_Functor_Lemma_I},
\begin{eqnarray*}
(N(V\oplus k)\otimes M^{(1)}(V\oplus k))_1 &=& N(V\oplus k)_1\otimes M^{(1)}(V\oplus k)_0 \\ &\cong& N(V\oplus k)_1\otimes M^{(1)}(V)
\end{eqnarray*}
and this isomorphism defines $\zeta_+$.

The isomorphism $\zeta_{-}:\bold{T}_{M^{(1)}}\circ \F \simeq \F \circ \bold{T}_{M^{(1)}}$ is simply given by the flip map.

\begin{theorem}\label{Commutativity_Theorem}
The data $(\bold{T}_{M^{(1)}}, \zeta_+,\zeta_{-})$ is a morphism of $\g'$-categorifications (cf. Remark \ref{gprimecateg}).
\end{theorem}

\begin{proof}
To prove this theorem we must check conditions (1)-(3) of Definition \ref{DefgMorphism}.  To begin, we check condition $(1)$, i.e. we must check that for $N,M\in\PP$ and $V\in\V$ the following diagram commutes:
$$
\xymatrix{
 & N(V)\otimes M^{(1)}(V) \ar[dl]_{\bold{T}_{M^{(1)}} \eta_1 } \ar[ddd]^{\eta_1 \bold{T}_{M^{(1)}}}\\ N(V\oplus k)_1\otimes V\otimes M^{(1)}(V) \ar[d]_{\zeta_{-}\E} & \\ N(V\oplus k)_1\otimes M^{(1)}(V)\otimes V \ar[dr]_{\F \zeta_{+}} & \\ & N(V\oplus k)_1\otimes M^{(1)}(V\oplus k)_0\otimes V
}
$$
Let  $n\otimes m\in N(V)\otimes M^{(1)}(V)$.  Applying the map $\bold{T}_{M^{(1)}} \eta_1$ we obtain $\sum_i (N(\tilde{\xi_i})n)_1\otimes e_i \otimes m$. Then applying $\zeta_{-}\E$ we obtain $\sum_i (N(\tilde{\xi_i})n)_1\otimes m\otimes e_i$.  Finally, applying  $\F \zeta_{+}$ we obtain $\sum_i (N(\tilde{\xi_i})n)_1\otimes M^{(1)}(\iota_V)m\otimes e_i$.    On the other hand,  starting with $n\otimes m$ and applying the map $\eta_1 \bold{T}_{M^{(1)}}$ we obtain $ \sum_i ((N\otimes M^{(1)})\oplus(\tilde{\xi_i})(n\otimes m))_1\otimes e_i$.
For every $i$, we have $((N\otimes M^{(1)})\oplus(\tilde{\xi_i})(n\otimes m))_1$ equals
$$(N(\tilde{\xi}_i)n)_0\otimes (M^{(1)}(\tilde{\xi}_i)m)_1+(N(\tilde{\xi}_i)n)_1\otimes (M^{(1)}(\tilde{\xi}_i)m)_0.$$
Since in the space $M^{(1)}(V\oplus k)$, all weights are multiples by $p$, it forces $ (M^{(1)}(\tilde{\xi}_i)m)_1=0$. Hence  we get
$$((N\otimes M^{(1)})\oplus(\tilde{\xi_i})(n\otimes m))_1=(N(\tilde{\xi}_i)n)_1\otimes (M^{(1)}(\tilde{\xi}_i)m)_0 $$
Then, by Lemma \ref{Xi_Lemma}, $(M^{(1)}(\tilde{\xi}_i)m)_0=M^{(1)}(\iota_V)m$. This proves condition (1).

To  check condition $(2)$ we have to show that the following diagram commutes:

$$
\xymatrix{
N(V\oplus k)_1\otimes M^{(1)}(V) \ar[r]^<<<<<{\zeta_+} \ar[d]_{\bold{T}_{M^{(1)}} X} & N(V\oplus k)\otimes M^{(1)}(V\oplus k)_0 \ar[d]^{X \bold{T}_{M^{(1)}}} \\ N(V\oplus k)_1\otimes M^{(1)}(V) \ar[r]^{\zeta_+} &  N(V\oplus k)\otimes M^{(1)}(V\oplus k)_0
}.
$$
Let $n\otimes m\in N(V\oplus k)_1\otimes M^{(1)}(V)$.  Applying $\bold{T}_{M^{(1)}} X$ and then $\zeta_+$, we obtain $\sum_{1\leq i\leq d-1} (x_{d,i}x_{i,d}-d+1)n\otimes M^{(1)}(\iota_V)m$. On the other hand, if we first apply $\zeta_+$ and then  $X \bold{T}_{M^{(1)}}$,  we also get $\sum_{1\leq i\leq d-1} (x_{d,i}x_{i,d}-d+1)(n\otimes M^{(1)}(\iota_V)m$), proving that condition 2 holds.
Proving condition (3) is a routine computation.
\end{proof}

Suppose as above that $[M]\mapsto b$ under $\varrho$.  We've seen   that $\bold{T}_{M^{(1)}}$ categorifies the twisted operator $b^{(1)}$.  We'd now like to categorify the adjoint operator $(b^{(1)})^*$.  A natural candidate for this is the functor $\bold{T}_{M^{(1)}}^*$, the right adjoint to $\bold{T}_{M^{(1)}}$.  The problem is that $\bold{T}_{M^{(1)}}^*$ is not exact, since $M^{(1)}$ is not projective.  Therefore we are led to consider $D^b(\PP)$, the bounded derived category of $\PP$.

\begin{lemma}
Let $M\in \PP$ and let $R\bold{T}_M^{*}$ be the right derived functor of $\bold{T}_M^*$. Then for any $N\in \PP$, $R\bold{T}^*_M(N)$ is a bounded complex in $D^b(\PP)$, and therefore $R\bold{T}_M^*$ defines a functor on $D^b(\PP)$.
\end{lemma}
\begin{proof}
By \cite[Theorem 2.10]{FS97} there exists  a projective resolution of $M$, $P^\cdot \twoheadrightarrow M$. We have a general formula : $$\bold{T}^*_M(N)(*)=Hom(M(\cdot),N(*\oplus \cdot)),  $$
By definition of derived functors, once we replace $M$ by $P^\cdot$, we obtain that $R\bold{T}^*_M(N)=(\bold{T}^*_P(N))^\cdot$. This formula does not depend on $N$, and hence in general $$R\bold{T}^*_M = (\bold{T}^*_P)^\cdot.$$
Since the global dimension of each $\PP_d$ is finite, on evaluation on $N\in \PP$,
$R\bold{T}^*_M(N)$ is  a complex whose cohomology vanishes almost everywhere.
Hence $R\bold{T}^*_M $ is defined on $D^b(\PP)$.
\end{proof}

Now we prove that the functors $R\bold{T}_{M^{(1)}}^*$ categorify $(b^{(1)})^*$. We first show that $\bold{T}^*_{M^{(1)}}$ is an endomorphism of  the $\g'$-categorification on $\PP$.

Let $\zeta_+^{\vee}:\bold{T}^*_{M^{(1)}}\circ\E \to \E\circ\bold{T}^*_{M^{(1)}}$ and  $\zeta_-^{\vee}:\bold{T}^*_{M^{(1)}}\circ\F \to \F\circ\bold{T}^*_{M^{(1)}}$ be the isomorphisms induced by the isomorphisms $\zeta_+,\zeta_-$ appearing in  Theorem \ref{Commutativity_Theorem}, the bi-adjunctions between $\E$ and $\F$, and also the adjunction between $\bold{T}_{M^{(1)}}$ and $\bold{T}^*_{M^{(1)}}$.

\begin{theorem}
\label{Adjunction_Commuting_Action}
Let $M\in\PP$.  The data $(\bold{T}^*_{M^{(1)}},\zeta_-^\vee, \zeta_+^\vee)$ induces a morphism of $\g'$-categorifications $\PP$. In particular  $R\bold{T}^*_{M^{(1)}}$ yields a $\g'$-endomorphism on $K(\PP)$.
\end{theorem}
\begin{proof}
As in the proof of Theorem \ref{Commutativity_Theorem}, one checks that the following diagram commutes:
$$
\xymatrix{
 & \bold{T}_{M^{(1)}} \ar[dl]_{\bold{T}_{M^{(1)}} \eta_2 } \ar[rd]^{\eta_2 \bold{T}_{M^{(1)}}} & \\
\bold{T}_{M^{(1)}}\E\F \ar[r]^<<<<<<{\zeta_{+}\F} & \E \bold{T}_{M^{(1)}}\F \ar[r]^<<<<<{\E \zeta_{-}} &   \E \F \bold{T}_{M^{(1)}}
}
$$
 By Propositions  \ref{Second_Compatibility} the conditions of Proposition \ref{Adjunction_Morphism_Categorification} are satisfied. Hence we can conclude that  $\bold{T}^*_{M^{(1)}}$ is an endomorphism of $\g'$-categorification on $\PP$. Then by Corollary \ref{Morphism_Representation}, $R\bold{T}^*_{M^{(1)}}$ yields an endomorphism of $K(\PP)$ commuting with the $\g'$ action.
\end{proof}

\begin{remark}
To properly discuss $\g$-categorifications on triangulated categories one can employ Rouquier's 2-representation theory \cite{R}.  In this setting, the functor $R\bold{T}^*_{M^{(1)}}$ should be a morphism of 2-representations of $\mathfrak{A}(\g')$, the 2-category associated to $\g'$ in loc. cit.

\end{remark}

\subsection{Categorifying the twist Heisenberg action}

We are ready now to categorify the twisted Heisenberg algebra action on Fock space.  Note that under our identification, $H^{(1)}$ is generated by operators $[T_M]^{(1)}$ and $([T_M]^{(1)})^*$.
\begin{lemma}
\label{Thm_veryweakcat}
The assignment
\begin{eqnarray*}
 & [\bold{T}_M]^{(1)}\mapsto& [\bold{T}_{M^{(1)}}]  \\
 & ([\bold{T}_{M}]^{(1)}) ^* \mapsto& [R\bold{T}_{M^{(1)}}^*]
\end{eqnarray*}
defines a representation of $H^{(1)}$ on $K(\PP)$ isomorphic to $B$.
\end{lemma}
\begin{proof}
By the adjunction of $\bold{T}_{M^{(1)}}$ and $R\bold{T}_{M^{(1)}}^*$ on $D^b(\PP)$, for any $M,L\in D^b(\PP)$,   we have
$$Hom_{D^b(\PP)}(\bold{T}_{M^{(1)}}N,L)\simeq Hom_{D^b(\PP)}(N,R\bold{T}_{M^{(1)}}^* L). $$
Under the identification of $K(D^b(\PP))=K(\PP)=B$, this descends to
$$([\bold{T}_{M^{(1)}}]([N]), [L])=([N], [R\bold{T}_{M^{(1)}}^* ][L]). $$
Therefore $[R\bold{T}_{M^{(1)}}^* ]=[\bold{T}_{M^{(1)}}]^*$. On the other hand, we have that $[\bold{T}_{M^{(1)}}]=[\bold{T}_{M}]^{(1)}$. Hence $[R\bold{T}_{M^{(1)}}^*]=([\bold{T}_{M}]^{(1)})^* $.
Since $[T_M]^{(1)}$ and their adjoints generate the algebra $H^{(1)}$, the proposition follows.
\end{proof}

Now we enrich the above lemma to a proper (weak) categorification.  In other words, we will distinguish endo-functors on $D^b(\PP)$ that lift the action of $H^{(1)}$ on $B$, \emph{and} functorial isomorphisms between these functors lifting defining relations of $H^{(1)}$.

Recall that for $M,N\in \PP$ we have a general formula
\begin{equation}
\bold T^*_M(N)(*)=Hom(M(\cdot),N(*\oplus \cdot)).
\end{equation}
It is quite easy to see that $^* \bold T_M\simeq\sharp \bold T^*_{M^\sharp} \sharp$.  Recall also that $R\bold T^*_M$ is the right derived functor of $\bold T^*_M$.  Then we have:
 \begin{equation}
R \bold T^*_M(N)(*)=RHom(M(\cdot),N(*\oplus \cdot)).
\end{equation}Let $\LT_M$ be the left derived functor of $^*\bold T_M$. We again have $\LT_M\simeq \sharp \RT_{M^\sharp} \sharp$.
\begin{lemma}
\label{Derived-Decomposition}
For $M,N,L \in \PP$ there is a canonical isomorphism
$$RHom(M,N\otimes L)\simeq RHom(\LT_N(M),L) .$$
\end{lemma}
\begin{proof}
The functor $^*\bold T_N$ sends projective objects to projective objects, since it admits an exact right adjoint $\bold T_N$. Then the right derived functor of $Hom(^*\bold T_N(\cdot),L)$ is isomorphic to $RHom(\LT_N(\cdot),L)$. Thus the lemma follows from the adjunction $$Hom(M,N\otimes L)\simeq Hom(^*\bold T_N(M),L). $$
\end{proof}

Recall that $\Gamma^n$ is the $n$-th divided power and $\Lambda^m$ is the $m$-th exterior power. Let $\Gamma^n{^\Frob}$ and $\Lambda^m{^\Frob}$ be the Frobenius twist of these functors.

\begin{proposition}
\label{Twist_Commuting_Computation}
On the derived category $D^b(\PP)$, we have isomorphisms of endo-functors
\begin{equation}
\label{Categorifying-Twist}
\RT_{\Gamma^n{^\Frob}}\circ \bold T_{\Lambda^m {^\Frob}}\simeq \bigoplus_{0 \leq j \leq min\{m,n\}} C(j) \otimes \bold T_{\Lambda^{m-j}{^\Frob}}\circ \RT_{\Gamma^{n-j}{^\Frob}}, \end{equation}
where $C(j)=RHom(\Lambda^j {^\Frob}, S^j {^\Frob})$ is a complex of vector spaces, which has only nonzero cohomologies at even degrees. Moreover the alternating sum of the dimension of all cohomologies of $C(j)$ is equal to $\binom{p}{j}$.
\end{proposition}

\begin{proof}

Given any $M\in \PP$ and $V\in \V$, then
$$\bold T_{\Lambda^m{^\Frob}}(\RT_{\Lambda^m{^\Frob}}(L))(V)= \Lambda^m(V{^\Frob})\otimes RHom(\Gamma^n {^\Frob}, L(V\oplus \cdot)). $$
Now consider the following chain of isomorphisms:
\begin{align*}
\RT_{\Gamma^n{^\Frob}} \bold T_{\Lambda^m {^\Frob}}(L)(V) &\simeq RHom(\Gamma^n {^\Frob}, \Lambda^m{^\Frob}(V\oplus \cdot)\otimes L(V\oplus \cdot) ) \\
  &\simeq  \bigoplus_{i+j=m}\Lambda^i(V^\Frob)\otimes RHom(\Gamma^n {^\Frob}, \Lambda^j{^\Frob} \otimes L(V\oplus \cdot))\\
&  \simeq \bigoplus_{i+j=m}  \Lambda^i(V{^\Frob})\otimes RHom(\LT_{\Lambda^j{^\Frob}}(\Gamma^n {^\Frob}),L(V\oplus \cdot) )\\
& \simeq \bigoplus_{0\leq j\leq min\{m,n\}}  \Lambda^{m-j}(V^\Frob)\otimes RHom(\LT_{\Lambda^j{^\Frob}}(\Gamma^n {^\Frob}),L(V\oplus \cdot) ).
\end{align*}
In the above chain of isomorphisms, the third isomorphism follows from Lemma \ref{Derived-Decomposition}, and the fourth isomorphism follows by degree consideration.
Now we want to compute $\LT_{\Lambda^j{^\Frob}}(\Gamma^n {^\Frob})$.
By the general formula $\LT_{M}\simeq \sharp \RT_{M^{\sharp}}\sharp$. In our case, since $(\Lambda^j{^\Frob})^\sharp=\Lambda^j{^\Frob}$, then we have $\LT_{\Lambda^j{^\Frob}}\simeq \sharp \RT_{\Lambda^j{^\Frob}}\sharp$. Moreover, $(\Gamma^n {^\Frob})^\sharp=S^n{^\Frob}$, so
$$\LT_{\Lambda^j{^\Frob}}(\Gamma^n {^\Frob})\simeq \sharp( \RT_{\Lambda^j{^\Frob}}(S^n{^\Frob})) $$
So we are reduced to compute  $\RT_{\Lambda^j{^\Frob}}(S^n{^\Frob})$. We claim it is isomorphic to $S^{n-j}{^\Frob}\otimes RHom(\Lambda^j{^\Frob},S^j{^\Frob} )$, since for any $W\in \V$,
 \begin{align*}
 \RT_{\Lambda^j{^\Frob}}(S^n{^\Frob})(W)& \simeq RHom(\Lambda^j{^\Frob}, S^n{^\Frob}(W\oplus \cdot)) \\
& \simeq \bigoplus _{k+\ell=n} S^k(W^\Frob)\otimes RHom(\Lambda^j{^\Frob},S^\ell{^\Frob})\\
& \simeq S^{n-j}(W^\Frob)\otimes RHom(\Lambda^j{^\Frob},S^j{^\Frob} ).
 \end{align*}
In the above chain of isomorphisms, the third isomorphism follows by degree considerations and the condition that $j\leq n$.

Therefore in the end we obtain $$\RT_{\Gamma^n{^\Frob}}\circ \bold T_{\Lambda^m {^\Frob}}\simeq \bigoplus_{0 \leq j \leq min\{m,n\}} C(j) \otimes \bold T_{\Lambda^{m-j}{^\Frob}}\circ \RT_{\Gamma^{n-j}{^\Frob}}, $$
where $C(j)=RHom(\Lambda^j {^\Frob}, S^j {^\Frob}) $. As to the computation of $C(j)$, i.e. the computation of $Ext^*(\Lambda^j {^\Frob}, S^j {^\Frob})$,  one can refer to \cite[Theorem 4.5]{FFSS}).

\end{proof}
From above proposition, we immediately obtain the following corollary.
\begin{corollary}
\label{Heisenberg_Level_p}
As linear operators on the space $B$ of symmetric functions,  for any $n,m\geq 0$ we have following equality
\begin{equation}
\label{Symmetric_Function_Twist}
(h_n^{(1)})^* e_m^{(1)}=\sum_{0\leq j \leq min\{m,n\}} \binom{p}{j}  e_{m-j}^{(1)} (h_{n-j}^{(1)})^* .
\end{equation}
\end{corollary}

\begin{theorem}
\label{TheoremCommutingAction}
The family of functors
$$
\{ \RT_{\Gamma^n{^\Frob}},\bold T_{\Lambda^n {^\Frob}}|n\geq0 \}
$$
along with the functorial isomorphisms from Proposition \ref{Twist_Commuting_Computation} (weakly) categorify the representation of $H^{(1)}$ on $B$.

\end{theorem}

\begin{proof}
By Proposition \ref{Twist_Commuting_Computation}, this family of functors categorifies the action of an algebra $\tilde{H}$ acting on $B$, where $\tilde{H}$ is the unital algebra generated by $t_n, s_m, (n,m=0,1,\cdots)$ subject to relations
\begin{align}
\label{Twist_Heisenberg_Presentation}
t_nt_m=t_mt_n, \text{     } s_ns_m=s_ms_n \\
\label{Twist_Heisenberg_Presentation_II}
t_ns_m=\sum_{j\leq min\{n,m\}} \binom{p}{j} s_{m-j}t_{n-j} ,
\end{align}
and $t_0=s_0=1$.  So to prove the theorem it remains to show that $\tilde{H} \cong H^{(1)}$.

By \cite[Section 2.2.8]{Lec}, the twisted Heisenberg $H^{(1)}$ has a presentation with standard generators $u_n,v_m$, subject to relations
$$u_nu_m=u_mu_n, \ v_nv_m=v_mv_n $$
$$v_nu_m-u_nv_m=p n\delta_{n,m}, $$
and $u_0=v_0=1$.  (In our realization of $H^{(1)}$, $u_n$ corresponds to the the $(pn)^{th}$ power symmetric function, and $v_n$ is its adjoint.)

By Corollary \ref{Heisenberg_Level_p} and Lemma \ref{Lemma_CommutingH},   there is a natural map  $\theta:\tilde{H}\to End_{\g'}(B)=H^{(1)}$ by mapping $s_n$ to $e_n^{(1)}$ and mapping $t_n$ to $(h_n^{(1)})^*$. In the algebra $B$ of symmetric functions, $\{e_n\}$, $\{h_n\}$, $\{p_n\}$ are three algebraically independent basis. So it is clear that $\theta$ is surjective; we will show it is an isomorphism.

We introduce a filtration $\tilde{F}$ on $\tilde{H}$, by declaring that the degree of $t_n$ is $n$ and the degree of $s_m$ is  $m$; similarly we introduce a filtration $F$ on $H^{(1)}$ such that the degree of $u_n$ is $n$ and the degree of $v_m$ is $m$. Since $\theta(t_n),\theta(s_n)^*$ are homogeneous symmetric functions, it is easy to see that $\theta$ preserves the filtration.  Since $\theta: \mathbb{C}[t_1,t_2,\cdots]\to \mathbb{C}[v_1,v_2,\cdots]$ and $\theta: \mathbb{C}[s_1,s_2,\cdots]\to \mathbb{C}[u_1,u_2,\cdots]$ are isomorphisms, and $gr_{\tilde{F}}(\tilde{H})$ and $gr_{F}(H^{(1)})$ are both polynomial rings,  it implies that $gr(\theta):gr_{\tilde{F}}(\tilde{H})\to gr_{F}(H^{(1)})$ is an isomorphism. Therefore $\theta:\tilde{H}\to H^{(1)}$ is an isomorphism.
\end{proof}

\begin{remark}
\label{Remarkglhat}
By Theorem \ref{HTYthm} we have a $\g$-categorification on $\PP$, which can be derived to obtain a $\g$-categorification on $D^b(\PP)$.  In particular, we have functors $E_i,F_i:D^b(\PP)\to D^b(\PP)$ and functorial isomorphisms between these functors that define a categorification of the Fock space representation of $\g$.  By the above theorem we also have a family of functors on $D^b(\PP)$, along with functorial isomorphisms, that categorify the action of $H^{(1)}$ on Fock space.  Combining all this data and  the commutativity data from Theorem \ref{Commutativity_Theorem} and Theorem \ref{Adjunction_Commuting_Action}, we therefore obtain a (weak) categorification of the (irreducible) Fock space representation of $\widehat{\mathfrak{gl}}_p$.
\end{remark}

\section{Appendix}
\label{Appendix}

In this appendix we gather some useful lemmas.  In Section \ref{SubSectionDuality} we record some relations between various dualities on the categories $\PP$ and $\R$.  In Section \ref{SubSectionCompat} we describe an explicit formula for an operator on the functor $\F$ obtained from $X$ via adjunction.  In Section \ref{SubSectionAdj} we prove that the adjoint of a morphism of $\g$-categorifications, suitably enriched, is again a morphism of $\g$-categorifications.

\subsection{Duality}
\label{SubSectionDuality}
In this part, we collect some facts relating to Kuhn duality $\sharp: \PP\to \PP$, a duality $\sharp': \R \to \R$ defined below,  the bi-adjoint pairs $(\E,\F)$, $(\E',\F')$, and the Schur-Weyl duality functor $\Schur:\PP\to \R$. We refer the reader to \cite{H} for proofs and more detail.

Recall that $\sharp: \PP \to \PP$ is defined as follows: for $M\in \PP$ and $V\in\V$, $M^{\sharp}(V)=M(V^*)^*$.

The functor $\sharp': \R \to \R$ is given as follows.  Firstly, for $J\in\X$, define $\theta:End_\X(J)\to End_\X(J)$ by $\theta(s_{i,j})=s_{i,j}$ and $\theta(w_1w_2)=\theta(w_2)\theta(w_1)$. Then, for $S\in \R$ and $J\in\X$, set $S^{\sharp'}(J)=S(J)^*$. Given any invertible map $f: J\to J$, set $S^{\sharp'}(f)=S(\theta(f))^*$.

Recall from Sections \ref{Section_gcatonP} and \ref{Section_gcatonR}  that we have bi-adjunction data
$$\xymatrix{ \I  \ar[r]^{\eta_1} & \F\E \ar[r]^{\epsilon_2} & \I},\xymatrix{ \I  \ar[r]^{\eta_2} & \E\F \ar[r]^{\epsilon_1} & \I}$$
$$\xymatrix{ \I  \ar[r]^{\eta'_1} & \F'\E' \ar[r]^{\epsilon'_2} & \I},\xymatrix{ \I  \ar[r]^{\eta_2'} & \E'\F' \ar[r]^{\epsilon'_1} & \I}$$
In Section \ref{SubSectionS_morphgcat} we defined isomorphisms $\zeta_+:\Schur \E\to \E'\Schur $, and $\zeta_-:\Schur \F\to \F'\Schur$.

Let $\sharp \E\simeq \E\sharp$, $\sharp \F\simeq \F \sharp$, $\sharp \E'\sharp'$, $\sharp' \F\simeq \F' \sharp$ and  $\sharp' \Schur \simeq \Schur \sharp $ be the obvious natural isomorphisms. So as not to inundate the reader with more notation, will not name these isomorphisms.  By composition we also obtain natural isomorphisms  $\sharp'\Schur \E \simeq \Schur \E \sharp$, $\sharp '\Schur \F\simeq \Schur \F \sharp$,  $\sharp \E \F \simeq \E \F \sharp$, $\sharp' \E' \F'\simeq \F' \E'$ etc...
\begin{lemma}\label{Duality_Lemma_1}
We have the following commutative diagrams
\begin{equation*}
\xymatrix{\sharp & \sharp \E \F  \ar[l]_{\sharp \eta_2} \ar[d]\\
& \E\F \sharp \ar[lu]^{\epsilon_1 \sharp}  ,& }
\xymatrix{\sharp \ar[r]^{\sharp \epsilon_2}  \ar[rd]_{\eta_1 \sharp} & \sharp \F \E \ar[d] \\
& \F\E \sharp .}
\end{equation*}

\begin{equation*}
\xymatrix{\sharp' & \sharp' \E' \F'  \ar[l]_{\sharp' \eta_2} \ar[d]\\
& \E'\F' \sharp' \ar[lu]^{\epsilon'_1 \sharp'}  ,& }
\xymatrix{\sharp' \ar[r]^{\sharp' \epsilon'_2}  \ar[rd]_{\eta'_1 \sharp'} & \sharp' \F' \E' \ar[d] \\
& \F'\E' \sharp' .}
\end{equation*}
\end{lemma}

\begin{lemma}\label{Duality_Lemma_2}
We have the following commutative diagrams
\begin{equation*}
\xymatrix{\sharp'\Schur \E  \ar[d] & \sharp' \E' \Schur \ar[l]_{\sharp' \zeta_+}\ar[d]\\ \Schur \E \sharp \ar[r]^{\zeta_+ \sharp}& \E'\Schur \sharp , &}
\xymatrix{\sharp'\Schur \F  \ar[d] & \sharp' \F' \Schur \ar[l]_{\sharp' \zeta_-}\ar[d]\\ \Schur \F \sharp \ar[r]^{\zeta_- \sharp}& \F'\Schur \sharp}
\end{equation*}
\end{lemma}

\subsection{Compatibility of operators $X$ and $Y$ via bi-adjunction}
\label{SubSectionCompat}
Let  $d\pi_{M,V}$  be the representation of $\mathfrak{gl}(V)$ on $M(V)$.  Then
\begin{equation}
\label{Derivation_Dual}
d\pi_{M^\sharp,V}(A)=d\pi_{M,V^*}(A^*)^*,
\end{equation}
for any $A\in \mathfrak{gl}(V)$.  Let $X^\sharp_{M,V}$ be the operator on $ \E(M)^\sharp(V)=\E(M)(V^*)^*$ induced from the operator $X_{M,V^*}$ on $\E(M)(V^*)$.

\begin{lemma}
\label{Dual_E}
We have the following commutative diagram:
$$\xymatrix{\E(M^\sharp) \ar[r]^{\alpha} \ar[d]^{X_{M^{\sharp}}} & \E(M)^\sharp  \ar[d]^{X^\sharp_M}  \\
             \E(M^\sharp) \ar[r]^{\alpha} &  \E(M)^\sharp }
$$
\end{lemma}
\begin{proof}
 Recall that $X_{M,V}=\sum_i d\pi_{M,V} (x_{n,i})d\pi_{M,V} (x_{i,n})-n$, where $n=dim V$. If we replace $V$ by $V^*$, then we have to exchange the role of $e_i$ and $\xi_i$. Hence $X^\sharp_{M,V}=\sum_i (d\pi_{M,V^*}(x_{i,n}^*)d\pi_{M,V^*}(x_{n,i}^*)) ^*-n$.  By (\ref{Derivation_Dual}) we have $X_{M^\sharp,V}=\sum_i d\pi_{M^\sharp,V}(x_{n,i})d\pi_{M^\sharp,V}(x_{i,n})-n=X_{M,V}^\sharp.$  \end{proof}

Let $Y$ be the operator on $\F$ induced from the operator $X$ on $\E$ via the adjunction $(\E,\F,\eta_1,\epsilon_1)$.  We describe the operator $Y$ explicitly in the following proposition (see \cite{H} for details).
\begin{proposition}
\label{Operator_Y}
Let $M\in \PP$ and $V\in \V$.  Choose a basis $e_i$ of $V$ and let $x_{ij}\in\mathfrak{gl}(V)$ be the operators $x_{ij}(e_k)=\delta_{jk}e_i$.  Then $Y_{M,V}=\sum_{i,j}x_{i,j}\otimes x_{j,i}$ on $\F(M)(V)=M(V)\otimes V$.
\end{proposition}

There is  a canonical isomorphism $\F(M^\sharp)\simeq \F(M)^\sharp $, which we use to identify the two spaces. Let $Y^\sharp_{M,V}$ be the induced operator on $\F(M)(V^*)^*$ from the operator $Y_{M,V^*}$ on $\F(M)(V^*)$.  Using the explicit formula for $Y$, one can prove as in Lemma \ref{Dual_E} the following:
\begin{lemma}
\label{Dual_F}
The operator $Y_{M,V}^\sharp$ coincides with $Y_{M^\sharp,V}$
\end{lemma}
\begin{proposition} \label{Second_Compatibility}
The operators $X$ on $\E$ and $Y$ on $\F$ are compatible via the second adjunction $(\F,\E,\eta_2,\epsilon_2)$.
\end{proposition}
\begin{proof}
Following the chain of natural isomorphisms
\begin{align*}
Hom(\F(M),N) & \simeq Hom(N^\sharp,\F(M)^\sharp) \simeq Hom(N^\sharp,\F(M^\sharp))\\
             & \simeq Hom(\E(N^\sharp),M^\sharp)
               \simeq Hom(M,\E(N^\sharp)^\sharp)\\
            & \simeq  Hom(M,\E(N)),
\end{align*}
 by Lemma \ref{Dual_E}, Proposition \ref{Operator_Y} and Lemma \ref{Dual_F}, one can conclude the proof of proposition.
\end{proof}

\subsection{Adjunction and morphisms of categorifictations}
\label{SubSectionAdj}

We first prove an abstract category-theoretic lemma.  Suppose we are in the situation where we have two categories $\C_i$, $i=1,2$.  Suppose we are also given  functors $E_i,F_i: \C_i\to \C_i$ and adjunctions $(E_i,F_i,\eta_i, \epsilon_i)$ for $i=1,2$. Moreover, assume $X_i$ is a natural transformation on $E_i$,  and let $Y_i$ be the induced natural transformation on $F_i$ via the adjunction $(E_i,F_i,\eta_i,\epsilon_i)$. Finally, let $\Phi:\C_1\to \C_2$ be a functor, along with natural isomorphisms $\zeta_+:\Phi E_1\simeq E_2\Phi$ and $\zeta_-:\Phi F_1\simeq F_2 \Phi$.

\begin{lemma}
\label{Abstract_Nonsense}
If the following diagrams commute
$$
\xymatrix{& \Phi \ar[rd]^{\eta_2 \Phi} \ar[dl]_{\Phi \eta_1}\\
\Phi F_1E_1 \ar[r]_{\zeta_-E_1} & F_2 \Phi E_1 \ar[r]_{F_2 \zeta_+} & F_2E_2 \Phi
}
$$
$$
\xymatrix{\Phi E_1 \ar[r]^{\zeta_+} \ar[d]^{\Phi X_1}    &E_2 \Phi \ar[d]^{X_2 \Phi}\\
\Phi E_1 \ar[r]^{\zeta_+}  &E_2 \Phi.}
$$
then this diagram also commutes
\begin{equation*}
\xymatrix{\Phi F_1 \ar[r]^{\zeta_-} \ar[d]^{\Phi Y_1}    &F_2 \Phi \ar[d]^{Y_2 \Phi}\\
\Phi F_1 \ar[r]^{\zeta_-}  &F_2 \Phi.}
\end{equation*}
\end{lemma}

\begin{proof}
We look at the following diagram,
\begin{equation*}
\xymatrix{&   F_2\Phi E_1F_1 \ar[rr]^{F_2 \zeta_+ F_1} \ar@{-->} [dd]^>>>>>>>>>>>>>>>>>{F_2 X_1 F_1} & & F_2E_2\Phi F_1 \ar[rr]^{F_2E_2 \zeta_-} \ar@{-->}[dd]^>>>>>>>>>>>>>>>>>>{F_2 X_2 \Phi F_1} & & F_2E_2 F_{2}\Phi \ar[dd]_{F_2X_2 F_2 \Phi} \\
\Phi F_1E_1F_1 \ar[ur]^{\zeta_- E_1F_1} \ar[dd]_{\Phi F_1 X_1 F_1} & & \Phi F_1 \ar[ll]^>>>>>>>>>>>>>>>>>>>>>>>>>{\Phi \eta_1 F_1} \ar[ur]^{\eta_2 \Phi F_1} \ar[rr]^>>>>>>>>>{\zeta_-} \ar[dd]^>>>>>>>>>>>>>>>>{\Phi Y_1} & & F_2\Phi \ar[ur]^{\eta_2 F_2 \Phi} \ar[dd]^>>>>>>>>>>>>>>>>{Y_2 \Phi}\\
& F_2\Phi E_1F_1 \ar@{-->}[rrrd]^<<<<<<<<<<<<<<<<<<<<<{F_2\Phi \epsilon_1}   \ar@{-->}[rr]^>>>>>>>>>>>>>>>>>>>>{F_2 \zeta_+ F_1} & & F_2E_2\Phi F_1  \ar@{-->}[rr]_>>>>>>>>>>>>>>>>>>>>{F_2E_2 \zeta_-} & & F_2E_2F_2\Phi \ar[dl]^{F_{2}\epsilon_2 \Phi} \\
\Phi F_1E_1F_1 \ar[rr]^{\Phi F_1 \epsilon_1} \ar@{-->}[ur]^{\zeta_- E_1 F_1}  & & \Phi F_1  \ar[rr]_{\zeta_-} & & F_2\Phi
}
\end{equation*}
We want to show the right square in the front face commute. It follows from the commutativity of other squares, which follows from our assumptions and functoriality. Note that we apply \cite[Lemma 5.3]{CR} to show the right triangle diagram  in the bottom face commute.

\end{proof}

\begin{lemma}\label{Miracle_Lemma}
Given a bi-adjunction $(E,F,\eta_1,\epsilon_1)$ and $(F,E,\eta_2,\epsilon_2)$,  the induced map $\eta_2^\vee: EF\to id$ from $\eta_2: id\to E F$ by adjunctions is equal to $\epsilon_1$.
\end{lemma}
\begin{proof}
By definition, the map $\eta_2 ^\vee$ is the composition of the following maps
$$\xymatrix{EF \ar[r]^{\eta_2EF} & EFEF \ar[r]^{E\epsilon_2F} & EF \ar[r]^{\epsilon_1} & id}. $$
Then note that $E\epsilon_2 \circ \eta_2E=1$, since $(F,E,\eta_2,\epsilon_2)$ is an adjunction. It is immediate that $\eta_2 ^\vee=\epsilon_1$.

\end{proof}

We would like now to prove the main result of the appendix, namely that the adjoint of a morphism of $\g$-categorifications is again a morphism of $\g$-categorifications. The proper formulation of this statement has to incorporate a fixed bi-adjunction between the functors $E,F$.

Suppose that $\C,\C'$ are two $\g$-categorifications with associated data $$(E,F,\eta_1,\epsilon_1, X, \sigma,\mathcal{C}=\oplus_\omega\mathcal{C}_\omega)$$ and $$(E',F',\eta_1',\epsilon_1', X', \sigma',\mathcal{C}'=\oplus_\omega\mathcal{C}'_\omega).$$ Moreover, suppose these data are enriched to include a bi-adjunction between $E,F$ and $E',F'$.  In other words, suppose we  fix $\eta_2,\epsilon_2$ and $\eta_2',\epsilon_2'$ as usual:
$$\xymatrix{ I  \ar[r]^{\eta_1} & FE \ar[r]^{\epsilon_2} & I},\xymatrix{ I  \ar[r]^{\eta_2} & EF \ar[r]^{\epsilon_1} & I}$$
$$\xymatrix{ I  \ar[r]^{\eta'_1} & F'E' \ar[r]^{\epsilon'_2} & I},\xymatrix{ I  \ar[r]^{\eta_2'} & E'F' \ar[r]^{\epsilon'_1} & I}$$

The bi-adjunction data are compatible in the following sense: let $X^\vee$ (resp. $X'^\vee $, $\sigma^\vee$, $\sigma'^\vee$) be the operator on $F$ (resp. $F'$, $F^2$, $F'^2$) induced from the operator $X$ (resp. $X'$,$\sigma$,$\sigma'$) on $E$ (resp. $E'$, $E^2$, $E'^2$) via the adjunction $(E,F,\eta_1,\epsilon_1)$
(resp. $(E',F',\eta_1',\epsilon'_1)$,...). Then we assume that the induced operator $(X^\vee)^\vee $ (resp. $(X'^\vee)^\vee$, $(\sigma^\vee)^\vee$,$(\sigma'^\vee) ^\vee$)  from $X^\vee$ (resp. $X'^\vee$,  $\sigma^\vee$, $\sigma'^\vee$) via the other adjunction $(F,E,\eta_2,\epsilon_2)$ (resp. $(F',E',\eta_2',\epsilon_2')$,...), coincides with $X$ (resp. $X'$,$\sigma$,$\sigma'$).

Now let $(\Phi,\zeta_+,\zeta_-):\C\to\C'$ be a morphism of $\g$-categorifications.  We assume that $\Phi$ preserves  the  bi-adjunction, i.e. the following diagram commutes
\begin{equation}\label{Additional_Diagram}
\xymatrix{
 & \Phi \ar[dl]_{\Phi \eta_2 } \ar[rd]^{\eta'_2 \Phi} & \\
\Phi EF \ar[r]^<<<<<<{\zeta_{+}F} & E' \Phi F \ar[r]^<<<<<{E' \zeta_{-}} &   E' F' \Phi.}
\end{equation}

\begin{proposition}
\label{Adjunction_Morphism_Categorification}
Let $\C,\C'$ be $\g$-categorifications with fixed compatible \emph{bi-adjunctions} as above.  Let $\Phi:\C\to\C'$ be a morphism of $\g$-categorifications satisfying (\ref{Additional_Diagram}).
Let $\Psi:\C'\to  \C$ be the right adjoint of $\Phi$. Let $\zeta^\vee _-: \Psi E' \to E \Psi$ and $\zeta^\vee_+: \Psi F'\to F \Psi$ be the induced isomorphisms from $\zeta_-$ and $\zeta_+$ by appropriate adjunctions. Then $(\Psi, \zeta_-^\vee, \zeta_+ ^\vee):\C'\to\C$ is a morphism from  the $\g$-categorifications. \end{proposition}
\begin{proof}
From Diagram \ref{Additional_Diagram}, by adjunction and Lemma \ref{Miracle_Lemma} we get the following commutative diagram
\begin{equation*}
\xymatrix{
 & \Psi  & \\
 \Psi E'F'  \ar[ur]^{\Psi \epsilon'_1}   \ar[r]^<<<<<<{\zeta_{-}^\vee F'} & E \Phi F' \ar[r]^<<<<<{E \zeta_{+}^\vee} &   E F \Phi  \ar[lu]_{\epsilon_1 \Psi}.}
\end{equation*}
By Lemma 5.3 in \cite{CR} this commutative diagram is equivalent to the commutativity of:
\begin{equation}
\label{New_Diagram_1}
\xymatrix{
 & \Psi \ar[dl]_{\Phi \eta'_1 } \ar[rd]^{\eta_1 \Psi} & \\
\Psi F'E' \ar[r]^<<<<<<{\zeta_{+}^\vee E'} & F \Psi E' \ar[r]^<<<<<{F \zeta_{-}^\vee} &   F E \Psi.}
\end{equation}

From the the commutativity of diagram $(1)$ and diagram $(2)$ in the definition of morphism of $\g$-categorifications (Definition \ref{DefgMorphism}),  by Lemma \ref{Abstract_Nonsense}  we get the following commutative diagram
\begin{equation*}
\xymatrix{\Phi F \ar[r]^{\zeta_-} \ar[d]^{X^\vee }    &F' \Phi \ar[d]^{X'^\vee \Phi}\\
\Phi F \ar[r]^{\zeta_-}  &F' \Phi.}
\end{equation*}
Then we apply appropriate adjunctions on this diagram, we get the following commutative diagram
\begin{equation}
\label{Adjunction_Diagram_2}
\xymatrix{\Psi E' \ar[r]^{\zeta_-^\vee} \ar[d]^{\Psi X}    &E \Psi \ar[d]^{X' \Psi}\\
\Psi E' \ar[r]^{\zeta_-^\vee}  & E \Psi.}
\end{equation}
By similar argument, we can get the following commutative diagram.
\begin{equation}
\label{Adjunction_Diagram_3}
\xymatrix{\Psi E'E' \ar[r]^{\zeta_-^\vee E'} \ar[d]^{\Psi \sigma'} & E \Psi E' \ar[r]^{E \zeta_-^\vee}     & E E\Psi \ar[d]^{\sigma \Psi}\\
\Psi E'E' \ar[r]^{\zeta_-^\vee E'} &  E\Psi E' \ar[r]^{E \zeta_-^\vee}   & EE \Psi.}
\end{equation}
From cite commutative diagrams (\ref{New_Diagram_1}), (\ref{Adjunction_Diagram_2}) and (\ref{Adjunction_Diagram_3}), the proposition follows.

\end{proof}

Jiuzu Hong,
School of Mathematical Sciences
Tel Aviv University,
Tel Aviv
69978, Israel.\\
\texttt{hjzzjh@gmail.com}
\medskip\\
Oded Yacobi,
Department of Mathematics,
University of Toronto,
Toronto, ON, M5S 2E4
Canada.\\
\texttt{oyacobi@math.toronto.edu}

\end{document}